\documentclass[a4paper,UKenglish]{lipics-v2018}

\usepackage{microtype}
\usepackage{tikz}
\usetikzlibrary{shapes,calc}

\makeatletter
\newcommand\footnoteref[1]{\protected@xdef\@thefnmark{\ref{#1}}\@footnotemark}
\makeatother

\title{On the Price of Independence for Vertex Cover, Feedback Vertex Set and Odd Cycle Transversal\footnote{A number of results in this paper have appeared in extended abstracts of the proceedings of MFCS~2018~\cite{DJPPZ18} and EuroComb~2019~\cite{DJPPZ19}.}}
\titlerunning{On the Price of Independence}

\author{Konrad K. Dabrowski}{Department of Computer Science, Durham University, UK}{konrad.dabrowski@durham.ac.uk}{https://orcid.org/0000-0001-9515-6945}{Supported by EPSRC (EP/K025090/1) and the Leverhulme Trust (RPG-2016-258).}
\author{Matthew Johnson}{Department of Computer Science, Durham University, UK}{matthew.johnson2@durham.ac.uk}{https://orcid.org/0000-0002-7295-2663}{Supported by the Leverhulme Trust (RPG-2016-258).}
\author{Giacomo Paesani}{Department of Computer Science, Durham University, UK}{giacomo.paesani@durham.ac.uk}{https://orcid.org/0000-0002-2383-1339}{}
\author{Dani\"el Paulusma}{Department of Computer Science, Durham University, UK}{daniel.paulusma@durham.ac.uk}{https://orcid.org/0000-0001-5945-9287}{Supported by EPSRC (EP/K025090/1) and the Leverhulme Trust (RPG-2016-258).}
\author{Viktor Zamaraev}{Department of Computer Science, Durham University, UK}{viktor.zamaraev@durham.ac.uk}{https://orcid.org/0000-0001-5755-4141}{Supported by EPSRC (EP/P020372/1).}

\authorrunning{K.\,K. Dabrowski, M. Johnson, G. Paesani, D. Paulusma and V. Zamaraev}

\Copyright{Konrad K. Dabrowski, Matthew Johnson, Giacomo Paesani, Dani\"el Paulusma, Viktor Zamaraev}

\subjclass{}


\EventEditors{Igor Potapov, Paul Spirakis, and James Worrell}
\EventNoEds{3}
\EventLongTitle{43rd International Symposium on Mathematical Foundations of Computer Science (MFCS 2018)}
\EventShortTitle{MFCS 2018}
\EventAcronym{MFCS}
\EventYear{2018}
\EventDate{August 27--31, 2018}
\EventLocation{Liverpool, GB}
\EventLogo{}
\SeriesVolume{117}
\ArticleNo{63}
\nolinenumbers
\hideLIPIcs

\theoremstyle{plain}
\newtheorem{openproblem}{Open Problem}

\newcommand{\NP}{{\sf NP}}

\def\dist{\qopname\relax n{dist}}
\def\deg{\qopname\relax n{deg}}
\def\fvs{\qopname\relax n{fvs}}
\def\ifvs{\qopname\relax n{ifvs}}
\def\vc{\qopname\relax n{vc}}
\def\ivc{\qopname\relax n{ivc}}
\def\oct{\qopname\relax n{oct}}
\def\ioct{\qopname\relax n{ioct}}
\newcommand{\qedllncs}{}

\begin{document}

\maketitle

\begin{abstract}
Let~$\vc(G)$, $\fvs(G)$ and~$\oct(G)$, respectively, denote the size of a minimum vertex cover, minimum feedback vertex set and minimum odd cycle transversal in a graph~$G$. One can ask, when looking for these sets in a graph, how much bigger might they be if we require that they are independent; that is, what is the {\em price of independence}?  If~$G$ has a vertex cover, feedback vertex set or odd cycle transversal that is an independent set, then we let~$\ivc(G)$, $\ifvs(G)$ or~$\ioct(G)$, respectively, denote the minimum size of such a set.
Similar to a recent study on the price of connectivity (Hartinger et al. EuJC 2016), we investigate for which graphs~$H$ the values of~$\ivc(G)$, $\ifvs(G)$ and~$\ioct(G)$ are bounded in terms of~$\vc(G)$, $\fvs(G)$ and~$\oct(G)$, respectively, when the graph~$G$ belongs to the class of $H$-free graphs.
We find complete classifications for vertex cover and feedback vertex set and an almost complete classification for odd cycle transversal (subject to three non-equivalent open cases). We also investigate for which graphs~$H$ the values of~$\ivc(G)$, $\ifvs(G)$ and~$\ioct(G)$ are equal to~$\vc(G)$, $\fvs(G)$ and~$\oct(G)$, respectively, when the graph~$G$ belongs to the class of $H$-free graphs. 
We find a complete classification for vertex cover and almost complete classifications for feedback vertex set (subject to one open case) and odd cycle transversal (subject to three open cases).
\end{abstract}

\section{Introduction}\label{sec:intro}

A transversal~$\tau(\pi)$ of a graph~$G$ is a set of vertices that transverse (intersect) all subsets of~$G$ that have some specific property~$\pi$.
The default aim is to find a transversal~$\tau(\pi)$ that has minimum size, but one may also add further conditions, such as demanding that the transversal must induce a connected subgraph or must be an {\it independent set} (set of pairwise non-adjacent vertices).
In this paper we focus on the latter condition and consider three classical and well-studied transversals obtained by specifying~$\pi$.

Let~$G$ be a graph.
We define the following three transversals of the vertex set~$V(G)$ of~$G$.
A set $S \subseteq V(G)$ is a \emph{vertex cover} if for every edge $uv \in E(G)$, at least one of~$u$ and~$v$ is in~$S$, or, equivalently, if the graph $G-S$ contains no edges.
A set $S \subseteq V(G)$ is a \emph{feedback vertex set} if for every cycle in~$G$, at least one vertex of the cycle is in~$S$, or, equivalently, if the graph $G-S$ is a forest.
A cycle is {\it odd} if it has an odd number of vertices.
A graph is {\it bipartite} if its vertex set can be partitioned into at most two independent sets.
A set $S \subseteq V(G)$ is an \emph{odd cycle transversal} if for every odd cycle in~$G$, at least one vertex of the cycle is in~$S$, or, equivalently, if $G-S$ is bipartite.
For each of these transversals, one usually wishes to investigate how small they can be and there is a vast research literature on this topic.

As mentioned, one can add an additional constraint: require the transversal to be an \emph{independent set}, that is, a set of vertices that are pairwise non-adjacent.
It may be the case that no such transversal exists under this constraint.
For example, a graph~$G$ has an independent vertex cover if and only if~$G$ is bipartite.
We are interested in the following research question:

\begin{quote}
{\em How is the minimum size of a transversal in a graph affected by adding the requirement that the transversal is independent?}
\end{quote}
Of course, this question can be interpreted in many ways.
For example, one might ask about the computational complexity of finding the transversals.
In this paper, we focus on the following: for the three transversals introduced above, is the size of a smallest possible independent transversal (assuming one exists) bounded in terms of the minimum size of a transversal?
That is, one might say, what is the \emph{price of independence}?

To the best of our knowledge, the term price of independence was first used by Camby~\cite{Camby17} in a recent unpublished manuscript.
She considered dominating sets of graphs (sets of vertices such that every vertex outside the set has a neighbour in the set).
As she acknowledged, though first to coin the term, she was building on past work.
In fact, Camby and her co-author Plein had given a forbidden induced subgraph characterization of those graphs~$G$ for which, for every induced subgraph of~$G$, there are minimum size dominating sets that are already independent~\cite{CP17}, and there are a number of further papers on the topic of the price of independence for dominating sets (see the discussion in~\cite{Camby17}).

We observe that this incipient work on the price of independence is a natural companion to recent work on the \emph{price of connectivity}, investigating the relationship between minimum size transversals and minimum size connected transversals (which, in contrast to independent transversals, will always exist for the transversals we consider, assuming the input graph is connected).
This work began with the work of Cardinal and Levy in their 2010 paper~\cite{CL10} and has since been taken in several directions; see, for example,~\cite{BHKP17,Ca19,CCFS14,CS14, CHJMP18, GS09,HJMP16}.

In this paper, as we broaden the study of the price of independence by investigating the three transversals defined above, we will concentrate on classes of graphs defined by a single forbidden induced subgraph~$H$, just as was done for the price of connectivity~\cite{BHKP17,HJMP16}.
That is, for a graph~$H$, we ask what, for a given type of transversal, is the price of independence in the class of $H$-free graphs?
The ultimate aim in each case is to find a dichotomy that allows us to say, given~$H$, whether or not the size of a minimum size independent transversal can be bounded in terms of the size of a minimum transversal.
We briefly give some necessary definitions and notation before presenting our results.

A {\em colouring} of a graph~$G$ is an assignment of positive integers (called {\em colours}) to the vertices of~$G$ such that if two vertices are adjacent, then they are assigned different colours.
A graph is {\em $k$-colourable} if there is a colouring that only uses colours from the set $\{1,\ldots,k\}$.
Equivalently, a graph is $k$-colourable if we can partition its vertex set into~$k$ (possibly empty) independent sets (called {\em colour classes} or {\em partition classes}).
For $s,t \geq 0$, let~$K_{s,t}$ denote the complete bipartite graph with partition classes of size~$s$ and~$t$, respectively (note that~$K_{s,t}$ is edgeless if $s=0$ or $t=0$).
For $r \geq 0$, the graph~$K_{1,r}$ is also called the $(r+\nobreak 1)$-vertex {\em star}; if $r \geq 2$ we say that the vertex in the partition class of size~$1$ is the {\em central vertex} of this star.
For $r \geq 1$, let~$K_{1,r}^+$ denote the graph obtained from~$K_{1,r}$ by subdividing one edge.
For $n \geq 1$, let~$P_n$ and~$K_n$ denote the path and complete graph on~$n$ vertices, respectively.
For $n \geq 3$, let~$C_n$ denote the cycle on~$n$ vertices.
The {\em disjoint union}~$G+\nobreak H$ of two vertex-disjoint graphs~$G$ and~$H$ is the graph with vertex set $V(G)\cup V(H)$ and edge set $E(G)\cup E(H)$.
We denote the disjoint union of~$r$ copies of a graph~$G$ by~$rG$.

\subparagraph*{The Price of Independence for Vertex Cover}
As mentioned above, a graph has an independent vertex cover if and only if it is bipartite.
For a bipartite graph~$G$, let~$\vc(G)$ denote the size of a minimum vertex cover, and let~$\ivc(G)$ denote the size of a minimum independent vertex cover.
Let~${\cal X}$ be a class of bipartite graphs.
Then~${\cal X}$ is {\it $\ivc$-bounded} if there exists a function $f:\mathbb{Z}_{\geq 0} \to \mathbb{Z}_{\geq 0}$ such that $\ivc(G) \leq f(\vc(G))$ for every $G \in {\cal X}$, and~${\cal X}$ is {\it $\ivc$-unbounded} if no such function exists, that is, if there is a~$k$ such that for every $s \geq 0$ there is a graph~$G$ in~${\cal X}$ with $\vc(G) \leq k$, but $\ivc(G) \geq s$.
Moreover,~${\cal X}$ is {\it $\ivc$-identical} if $\ivc(G)=\vc(G)$ for every $G \in {\cal X}$.

In our first two results, proven in Section~\ref{sec:vc}, we determine for every graph~$H$, whether or not the class of $H$-free bipartite graphs is $\ivc$-bounded or $\ivc$-identical, respectively.

\begin{theorem}\label{thm:indep-vc}
Let~$H$ be a graph.
The class of $H$-free bipartite graphs is $\ivc$-bounded if and only if~$H$ is an induced subgraph of $K_{1,r}+\nobreak rP_1$ or~$K_{1,r}^+$ for some~$r\geq 1$.
\end{theorem}

\begin{theorem}\label{t-vci}
Let~$H$ be a graph.
The class of $H$-free bipartite graphs is $\ivc$-identical if and only if~$H$ is an induced subgraph of~$K_{1,3}^+$ or~$2P_1+\nobreak P_3$.
\end{theorem}

\subparagraph*{The Price of Independence for Feedback Vertex Set}
A graph has an independent feedback vertex set if and only if its vertex set can be partitioned into an independent set and a set of vertices that induces a forest; graphs that have such a partition are said to be \emph{near-bipartite}.
In fact, minimum size independent feedback vertex sets have been the subject of much research from a computational perspective: to find such a set is \NP-hard in general, but there are fixed-parameter tractable algorithms and polynomial-time algorithms for certain graph classes; we refer to~\cite{BDFJP17} for further details.
For a near-bipartite graph~$G$, let~$\fvs(G)$ denote the size of a minimum feedback vertex set, and let~$\ifvs(G)$ denote the size of a minimum independent feedback vertex set.
Given a class~${\cal X}$ of near-bipartite graphs, we say that~${\cal X}$ is {\it $\ifvs$-bounded} if there is a function $f:\mathbb{Z}_{\geq 0} \to \mathbb{Z}_{\geq 0}$ such that $\ifvs(G) \leq f(\fvs(G))$ for every $G \in {\cal X}$ and {\it $\ifvs$-unbounded} otherwise.
Moreover, a class~${\cal X}$ of near-bipartite graphs is {\it $\ifvs$-identical} if $\ifvs(G) = \fvs(G)$ for every $G \in {\cal X}$.

In our next two results, proven in Section~\ref{sec:fvs}, we almost completely determine for every graph~$H$, whether or not the class of $H$-free near-bipartite graphs is $\ifvs$-bounded or $\ifvs$-identical, respectively; the only open case left is determining whether the class of $K_{1,3}$-free near-bipartite graphs is $\ifvs$-identical.

\begin{theorem}\label{thm:indep-fvs}
Let~$H$ be a graph.
The class of $H$-free near-bipartite graphs is $\ifvs$-bounded if and only if~$H$ is isomorphic to $P_1+\nobreak P_2$, a star or an edgeless graph.
\end{theorem}

\begin{theorem}\label{t-fvsi}
Let~$H$ be a graph different from~$K_{1,3}$.
The class of $H$-free near-bipartite graphs is $\ifvs$-identical if and only if~$H$ is a (not necessarily induced) subgraph of~$P_3$.
\end{theorem}

\subparagraph*{The Price of Independence for Odd Cycle Transversal}
A graph has an independent odd cycle transversal~$S$ if and only if it has a $3$-colouring, since, by definition, we are requesting that~$S$ is an independent set of~$G$ such that $G-S$ has a $2$-colouring.
For a $3$-colourable graph~$G$, let~$\oct(G)$ denote the size of a minimum odd cycle transversal, and let~$\ioct(G)$ denote the size of a minimum independent odd cycle transversal.
Given a class~${\cal X}$ of $3$-colourable graphs, we say that~${\cal X}$ is {\it $\ioct$-bounded} if there is a function $f:\mathbb{Z}_{\geq 0} \to \mathbb{Z}_{\geq 0}$ such that $\ioct(G) \leq f(\oct(G))$ for every $G \in {\cal X}$ and {\it $\ioct$-unbounded} otherwise.
Moreover, a class~${\cal X}$ of $3$-colourable graphs is {\it $\ioct$-identical} if $\ioct(G) = \oct(G)$ for every $G \in {\cal X}$.

In our final two results, proven in Section~\ref{sec:oct}, we address the question of whether or not, for a graph~$H$, the class of $H$-free $3$-colourable graphs is $\ioct$-bounded or $\ioct$-identical, respectively.
Here, we do not have complete dichotomies.
For the former question, we prove that the number of non-equivalent open cases left is three, namely the cases when $H\in \{K_{1,4}, K_{1,3}^+,K_{1,4}^+\}$.
Note that for the latter question there are also three missing cases.

\begin{theorem}\label{thm:indep-oct}
Let~$H$ be a graph.
The class of $H$-free $3$-colourable graphs is $\ioct$-bounded:
\begin{itemize}
\item if~$H$ is an induced subgraph of~$P_4$ or $K_{1,3}+\nobreak sP_1$ for some $s\geq 0$ and
\item only if~$H$ is an induced subgraph of~$K_{1,4}^+$ or $K_{1,4}+\nobreak sP_1$ for some $s\geq 0$.
\end{itemize}
\end{theorem}

\begin{theorem}\label{t-octi}
Let~$H$ be a graph such that $H \notin \{K_{1,3}, K_{1,3}^+, 2P_1+\nobreak P_3\}$.
The class of $H$-free $3$-colourable graphs is $\ioct$-identical if and only if~$H$ is a (not necessarily induced) subgraph of~$P_4$ that is not isomorphic to~$2P_2$.
\end{theorem}

\subparagraph*{Further Notation}
Let~$G$ be a graph.
For a vertex $v \in V(G)$, the {\em (open) neighbourhood}~$N(v)$ of~$v$ is the set of vertices adjacent to~$v$; the {\em closed neighbourhood}~$N[v]$ is defined to be $N(v) \cup \{v\}$.
We let~$\deg(v)$ denote the {\em degree} of~$v$, i.e. the number of neighbours of~$v$.
Let~$S$ be a set of vertices in~$G$.
We let~$N(S)$ denote the set of those vertices outside~$S$ that have at least one neighbour in~$S$, i.e. $N(S) = \{v \in V(G)\setminus S \; | \; \exists u \in S, uv \in E(G)\}$.
We also let~$N[S]$ denote the set $N(S) \cup S$, and~$G[S]$ denote the \emph{subgraph of~$G$ induced by~$S$}, that is, the graph with vertex set~$S$, where two vertices in~$S$ in are adjacent~$G[S]$ if and only if they are adjacent in~$G$.
Furthermore, we let $G-S$ denote the graph obtained from~$G$ by removing all vertices in~$S$, that is,~$G-S = G[V(G) \setminus S]$.
For a graph~$F$, we write $F\subseteq_i G$ if~$F$ is an induced subgraph of~$G$ and we write $F\subseteq G$ if~$F$ is a (not necessarily induced) subgraph of~$G$, that is, $V(F)\subseteq V(G)$ and $E(F)\subseteq E(G)$.

Let $\{H_1, \ldots, H_s\}$ be a set of graphs.
Then a graph~$G$ is said to be {\em $(H_1, \ldots, H_s)$-free} if it contains no induced subgraph isomorphic to a graph in~$\{H_1, \ldots, H_s\}$.
If $s=1$, then we simply write $H_1$-free instead of $(H_1)$-free.

A vertex $v \in V(G)$ is \emph{complete} (resp. \emph{anti-complete}) to a set $X \subseteq V(G)$ if~$v$ is adjacent (resp. non-adjacent) to every vertex in~$X$.
A set $Y \subseteq V(G)$ is \emph{complete} (resp. \emph{anti-complete}) to a set $X \subseteq V(G)$ if every vertex of~$Y$ is complete (resp. anti-complete) to~$X$.
A graph is {\em complete multi-partite} if its vertex set can be partitioned into independent sets that are complete to each other.
A vertex $v \in V(G)$ is {\em dominating} if it is complete to $V(G) \setminus \{v\}$.
For two vertices $u,v \in V(G)$, we let $\dist(u,v)$ denote the length of a shortest path from~$u$ to~$v$ (by convention, $\dist(u,v)=\infty$ if~$u$ and~$v$ are in different connected components of~$G$).

The {\em complement}~$\overline{G}$ of a graph~$G$ has the same vertex set as~$G$ and an edge between two distinct vertices if and only if these vertices are not adjacent in~$G$.
The {\it double star}~$D_{p,q}$ is the tree on vertices $x,y,u_1,\ldots u_p,v_1,\ldots,v_q$ with edges~$xy$,~$xu_i$ for $i \in \{1,\dots,p\}$, and~$yv_j$ for $j\in \{1,\ldots,q\}$ (see \figurename~\ref{fig:Dsr-for-vc} for an example).

\section{Vertex Cover}\label{sec:vc}

In this section we prove Theorems~\ref{thm:indep-vc} and~\ref{t-vci} as part of a more general theorem.
We start with a useful lemma.

\begin{lemma}\label{lem:star+indep-free-bip-structure}
Let $r,s \geq 1$.
If~$G$ is a $(K_{1,r}+\nobreak sP_1)$-free bipartite graph with bipartition~$(X,Y)$ such that $|X|,|Y| \geq rs+r-1$, then either:
\begin{itemize}
\item every vertex of~$G$ has degree less than~$r$ or
\item fewer than~$s$ vertices of~$X$ have more than $s-\nobreak 1$ non-neighbours in~$Y$ and fewer than~$s$ vertices of~$Y$ have more than $s-\nobreak 1$ non-neighbours in~$X$.
\end{itemize}
\end{lemma}

\begin{proof}
Let~$G$ be a $(K_{1,r}+\nobreak sP_1)$-free bipartite graph with bipartition~$(X,Y)$ such that $|X|,|Y| \geq rs+r-1$.
No vertex in~$X$ can have both~$r$ neighbours and~$s$ non-neighbours in~$Y$, otherwise~$G$ would contain an induced $K_{1,r}+\nobreak sP_1$.
Therefore every vertex in~$X$ has degree either at most~$r-\nobreak 1$ or at least $|Y|-(s-1) \geq rs+r-s$.
By symmetry, we may assume that there is a vertex $x \in X$ of degree at least~$r$.
Suppose, for contradiction, that there is a set $X' \subseteq X$ of~$s$ vertices, each of which has more than $s-\nobreak 1$ non-neighbours in~$Y$.
Then every vertex of~$X'$ has degree at most~$r-\nobreak 1$.
Since $\deg(x) \geq rs+r-s= s(r-1)+r$, there must be a set $Y' \subseteq N(x)$ of~$r$ neighbours of~$x$ that have no neighbours in~$X'$.
Then $G[\{x\} \cup Y' \cup X']$ is a $K_{1,r}+\nobreak sP_1$, a contradiction.
It follows that fewer than~$s$ vertices in~$X$ have more than $s-\nobreak 1$ non-neighbours in~$Y$.
Since $|X|\geq r+(s-1)$, there is a set $X'' \subsetneq X$ of~$r$ vertices, each of which has at most $s-\nobreak 1$ non-neighbours in~$Y$.
Since $|Y|>r(s-1)$, there must be a vertex $y \in Y$ that is complete to~$X''$, and therefore has $\deg(y) \geq r$.
Repeating the above argument, it follows that fewer than~$s$ vertices of~$Y$ have more than $s-\nobreak 1$ non-neighbours in~$X$.
This completes the proof.\qedllncs
\end{proof}

We recall that a graph has an independent vertex cover if and only if it is bipartite, and we prove two more lemmas.

\begin{lemma}\label{lem:vc-K1r+sP1}
Let $r,s \geq 1$.
If~$G$ is a $(K_{1,r}+\nobreak sP_1)$-free bipartite graph, then $\ivc(G) \leq r\cdot \vc(G)+rs$.
\end{lemma}

\begin{proof}
Let~$G$ be a $(K_{1,r}+\nobreak sP_1)$-free bipartite graph.
Fix a bipartition~$(X,Y)$ of~$G$.
Let~$S$ be a minimum vertex cover of~$G$, so $|S|=\vc(G)$.
We may assume that $\vc(G) \geq\nobreak 2$, otherwise $\ivc(G)=\vc(G)$, in which case we are done.
We may also assume that $|X|,|Y| > \vc(G)r+rs > rs+r-1$, otherwise~$X$ or~$Y$ is an independent vertex cover of the required size, and we are done.
If every vertex of~$G$ has degree at most~$r-1$, then $S'=(S \cap Y)\cup (N(S \cap X))$ is an independent vertex cover in~$G$ of size at most~$\vc(G)(r-1)$, and we are done.
By Lemma~\ref{lem:star+indep-free-bip-structure}, we may therefore assume that fewer than~$s$ vertices of~$X$ have more than $s-1$ non-neighbours in~$Y$.
We will show that this leads to a contradiction.
Since $|X|,|Y| \geq \vc(G)+s$, there must be a set~$S'$ of $\vc(G)+\nobreak 1$ vertices in~$X$ that each have at least $\vc(G)+\nobreak 1$ neighbours in~$Y$.
If a vertex $x \in V(G)$ has degree at least $\vc(G)+\nobreak 1$, then $|N(x)|>|S|$, so $x \in S$.
Therefore every vertex of~$S'$ must be in~$S$, contradicting the fact that $|S'| =\vc(G)+\nobreak 1>\vc(G)=|S|$.
\end{proof}

\begin{lemma}\label{lem:vc-K1r+}
Let $r \geq 2$.
If~$G$ is a $K_{1,r}^+$-free bipartite graph, then $\ivc(G) \leq (r-1)(\vc(G))^2$.
\end{lemma}

\begin{proof}
Clearly it is sufficient to prove the lemma for connected graphs~$G$.
Let~$G$ be a connected $K_{1,r}^+$-free bipartite graph.
Fix a bipartition $(X,Y)$ of~$G$.
Let~$S$ be a minimum vertex cover of~$G$, so $|S|=\vc(G)$.
We may assume that $\vc(G) \geq 2$, otherwise $\ivc(G)=\vc(G)$ and we are done.
We may also assume that $|X|,|Y| > (\vc(G))^2(r-1)$, otherwise~$X$ or~$Y$ is an independent vertex cover of the required size.

If there are two vertices $x,y \in X$ with $\dist(x,y)=2$ and $\deg(x)\geq \deg(y)+\nobreak (r-1)$, then $x,y$, a common neighbour of~$x$ and~$y$, and~$r-1$ vertices from $N(x) \setminus N(y)$ would induce a~$K_{1,r}^+$ in~$G$, a contradiction.
Therefore, if $x,y \in X$ with $\dist(x,y)=2$, then $|\deg(x)-\deg(y)| \leq r-2$, so $|\deg(x)-\deg(y)| \leq (\frac{r-2}{2})\dist(x,y)$.
By the triangle inequality and induction, it follows that if $x,y \in X$, then $|\deg(x)-\deg(y)| \leq (\frac{r-2}{2})\dist(x,y)$.
Observe that $\vc(P_{2\vc(G)+2})=\vc(G)+1$, so~$G$ must be $P_{2\vc(G)+2}$-free.
Since~$G$ is connected, it follows that if $x,y \in V(G)$, then $\dist(x,y) < 2\vc(G)+1$.
We conclude that if $x,y \in X$, then $|\deg(x)-\deg(y)| \leq \vc(G)(r-2)$.
Note that if a vertex $x \in V(G)$ has degree at least $\vc(G)+\nobreak 1$, then $|N(x)|>|S|$ and so $x \in S$.

Since $|X| > (\vc(G))^2(r-1) > \vc(G) = |S|$, there must be a vertex $y \in X \setminus S$.
Since $y \in X \setminus S$, it follows that $\deg(y)\leq \vc(G)$.
It follows that $\deg(x)\leq \deg(y)+\vc(G)(r-2) \leq \vc(G)(r-1)$ for all $x \in X$.
We conclude that $S'=(S \cap Y)\cup (N(S \cap X))$ is an independent vertex cover in~$G$ of size at most~$(\vc(G))^2(r-1)$.
This completes the proof.\qedllncs
\end{proof}

A graph is an {\it almost complete bipartite graph} if it can be obtained from a complete bipartite graph by removing a (possibly empty) set of edges that form a matching.
We need the following lemma due to Alekseev.

\begin{lemma}[\cite{Al04}]\label{l-alek}
Every connected $K_{1,3}^+$-free bipartite graph is either a path, a cycle or an almost complete bipartite graph.
\end{lemma}

We also need the following lemma.

\begin{lemma}\label{almost-complete-bipartite}
Let~$G$ be an almost complete bipartite graph.
Then $\ivc(G)=\vc(G)$.
\end{lemma}

\begin{proof}
Notice that $\ivc(G)=\vc(G)$ holds if and only if the equality holds for every connected component of~$G$.
Therefore, without loss of generality, we may assume that~$G$ is connected.
Let $X,Y$ be the parts of the bipartition of~$G$, and let~$S$ be a minimum vertex cover of~$G$.
We may assume without loss of generality that $|X|\leq |Y|$.
If $\vc(G)\leq 1$, then $\ivc(G)=\vc(G)$.
Therefore we may assume that $|X|\geq \vc(G)\geq 2$.
If~$S$ is independent or $|S|=|X|$, then again $\ivc(G)=\vc(G)$.

Now we assume that~$S$ is not independent and $|X|>|S|$.
This implies that there exist two adjacent vertices $x\in X\cap S$ and $y\in Y\cap S$, and another vertex $y'\in Y\setminus S$.
Since~$G$ is a connected almost complete bipartite graph, the vertex~$y'$ is adjacent to all vertices of~$X$ but at most one.
Moreover, since $y'\not\in S$, the neighbourhood of~$y'$ is contained in~$S$.
Therefore $|X|>|S|\geq |\{y\}\cup N(y')| \geq 1 + (|X|-1)=|X|$, a contradiction.
\end{proof}

Our next theorem is the main result of this section and immediately implies Theorems~\ref{thm:indep-vc} and~\ref{t-vci}.
If an upper bound given in this theorem is tight, that is, if there exists an $H$-free bipartite graph~$G$ for which equality holds, we indicate this by a~$*$ in the corresponding row (whereas the other upper bounds are not known to be tight).

\begin{theorem}\label{t-vc}
Let~$H$ be a graph.
Then the following two statements hold:
\begin{enumerate}[(i)]
\renewcommand{\theenumi}{(\roman{enumi})}
\renewcommand\labelenumi{(\roman{enumi})}
\item\label{t-vc:i} the class of $H$-free bipartite graphs is $\ivc$-bounded if and only if~$H$ is an induced subgraph of $K_{1,r}+\nobreak rP_1$ or~$K_{1,r}^+$ for some~$r\geq 1$; and
\item\label{t-vc:ii} the class of $H$-free bipartite graphs is $\ivc$-identical if and only if~$H$ is an induced subgraph of~$K_{1,3}^+$ or~$2P_1+\nobreak P_3$.\\[-8pt]
\end{enumerate}
In particular, the following statements hold for every $H$-free bipartite graph~$G$: 
	\begin{enumerate}[(1)*]
\renewcommand{\theenumi}{(\arabic{enumi})}
\renewcommand\labelenumi{(\arabic{enumi})\phantom{*}}

                \refstepcounter{enumi}
		\item[\textcolor{darkgray}{\sffamily\bfseries\upshape\mathversion{bold}{(1)*}}]\label{t-vc:1} $\ivc(G) = \vc(G)$ if $H \subseteq_i K_{1,3}^+$ or $H \subseteq_i 2P_1+\nobreak P_3$
                \refstepcounter{enumi}
		\item[\textcolor{darkgray}{\sffamily\bfseries\upshape\mathversion{bold}{(2)*}}]\label{t-vc:2} $\ivc(G) \leq \vc(G) + 1$ if $H = K_{1,3} + P_1$
		\item\label{t-vc:3} $\ivc(G) \leq \vc(G) + s - 3$ if $H = sP_1$ for $s \geq 5$
		\item\label{t-vc:4b} $\ivc(G) \leq \vc(G) + s - 2$ if $H = sP_1+\nobreak P_2$ for $s \geq 3$
                \refstepcounter{enumi}
		\item[\textcolor{darkgray}{\sffamily\bfseries\upshape\mathversion{bold}{(5)*}}]\label{t-vc:4} $\ivc(G) \leq \vc(G) + s - 2$ if $H = sP_1+\nobreak P_3$ for $s \geq 3$
		\item\label{t-vc:5} $\ivc(G) \leq \vc(G) + 3s + 2$ if $H = K_{1,3} +\nobreak sP_1$ for $s \geq 2$
                \refstepcounter{enumi}
		\item[\textcolor{darkgray}{\sffamily\bfseries\upshape\mathversion{bold}{(7)*}}]\label{t-vc:6} $\ivc(G) \leq (r-1) \vc(G) - 1$ if $H = K_{1,r}$ for $r \geq 4$
		\item\label{t-vc:7} $\ivc(G) \leq r\cdot \vc(G) + rs$ if $H = K_{1,r} +\nobreak sP_1$ for $r \geq 4, s \geq 1$
		\item\label{t-vc:8} $\ivc(G) \leq (r-1) \vc(G)^2$ if $H = K_{1,r}^+$ for $r \geq 4$
	\end{enumerate}
\end{theorem}

\begin{proof}
We start by proving~{\sffamily\bfseries\upshape\mathversion{bold}{\ref{t-vc:i}}}.

\subparagraph*{\ref{t-vc:i}: ``$\Leftarrow$''.}
First suppose that~$H$ is an induced subgraph of $K_{1,r}+\nobreak rP_1$ or~$K_{1,r}^+$ for some~$r$, then Lemma~\ref{lem:vc-K1r+sP1} or~\ref{lem:vc-K1r+}, respectively, implies that the class of $H$-free bipartite graphs is $\ivc$-bounded.

\subparagraph*{\ref{t-vc:i}: ``$\Rightarrow$''.}
Now suppose that the class of $H$-free bipartite graphs is $\ivc$-bounded, that is, there is a function $f:\mathbb{Z}_{\geq 0} \to \mathbb{Z}_{\geq 0}$ such that $\ivc(G) \leq f(\vc(G))$ for all $H$-free bipartite graphs~$G$.
We will show that~$H$ is an induced subgraph of $K_{1,r}+\nobreak rP_1$ or~$K_{1,r}^+$ for some~$r$.

For $r \geq 1$, $s \geq 2$, let~$D_s^r$ denote the graph formed from~$2K_{1,s}$ and~$P_{2r}$ by identifying the two end-vertices of the~$P_{2r}$ with the central vertices of the respective~$K_{1,s}$'s (see also \figurename~\ref{fig:Dsr-for-vc}; note that $D_s^1=D_{s,s}$).
It is easy to verify that $\vc(D_s^r)=r+1$ and $\ivc(D_s^r)=r+s$.
Note that, for every $r\geq 1$, $$\ivc(D_{f(r+1)}^r)=r+f(r+1)=r+f(\vc(D_{f(r+1)}^r))>f(\vc(D_{f(r+1)}^r)).$$
Hence, for every $r \geq 1$, $D_{f(r+1)}^r$ cannot be $H$-free.
Note that for $r \geq 1$ and $s,t \geq 2$, if $s \leq t$ then~$D_s^r$ is an induced subgraph of~$D_t^r$.
Therefore, for each $r \geq 1$, there must be an~$s$ such that~$D_s^r$ is not $H$-free.
In other words, for each $r \geq 1$, $H$ must be an induced subgraph of~$D_s^r$ for some~$s$.

\begin{figure}
\begin{center}
\begin{subfigure}{0.33\textwidth}
\begin{tikzpicture}[scale=0.7]
\coordinate (x1) at (0,0);
\coordinate (x2) at (1,0);
\coordinate (x3) at (2,0);
\coordinate (x4) at (4,0);
\coordinate (y1) at (1,3);
\coordinate (y2) at (3,3);
\coordinate (y3) at (4,3);
\coordinate (y4) at (5,3);

\draw (x1) -- (y1) -- (x2) (x3) -- (y1) -- (x4) -- (y4) (y2) -- (x4) -- (y3);

\draw[fill=black](x1) circle [radius=3pt];
\draw[fill=black](x2) circle [radius=3pt];
\draw[fill=black](x3) circle [radius=3pt];
\draw[fill=black](x4) circle [radius=3pt];
\draw[fill=white](y1) circle [radius=3pt];
\draw[fill=white](y2) circle [radius=3pt];
\draw[fill=white](y3) circle [radius=3pt];
\draw[fill=white](y4) circle [radius=3pt];

\end{tikzpicture}
\end{subfigure}
\qquad
\begin{subfigure}{0.33\textwidth}
\begin{tikzpicture}[scale=0.7]
\coordinate (x1) at (0,0);
\coordinate (x2) at (1,0);
\coordinate (x3) at (2.5,0);
\coordinate (x4) at (4.5,0);
\coordinate (y1) at (0.5,3);
\coordinate (y2) at (2.5,3);
\coordinate (y3) at (4,3);
\coordinate (y4) at (5,3);

\draw (x1) -- (y1) -- (x2) (y1) -- (x3) -- (y2) -- (x4) (y3) -- (x4) -- (y4);

\draw[fill=black](x1) circle [radius=3pt];
\draw[fill=black](x2) circle [radius=3pt];
\draw[fill=black](x3) circle [radius=3pt];
\draw[fill=black](x4) circle [radius=3pt];
\draw[fill=white](y1) circle [radius=3pt];
\draw[fill=white](y2) circle [radius=3pt];
\draw[fill=white](y3) circle [radius=3pt];
\draw[fill=white](y4) circle [radius=3pt];

\end{tikzpicture}
\end{subfigure}
\end{center}
\caption{The graphs~$D_3^1=D_{3,3}$ and~$D_2^2$.
The black vertices form a minimum independent vertex cover.}
\label{fig:Dsr-for-vc}
\end{figure}
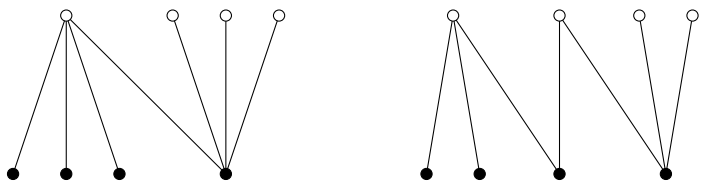

In particular, the above means that we may assume that~$H$ is an induced subgraph of~$D_t^1$ for some $t \geq 1$.
If~$H$ contains at most one of the central vertices of the stars that form the~$D_t^1$, then~$H$ is an induced subgraph of~$K_{1,t}+\nobreak tP_1$ and we are done, so we may assume~$H$ contains both central vertices.
If one of these central vertices has at most one neighbour that is not a central vertex, then~$H$ is an induced subgraph of~$K_{1,t+1}^+$, and we are done.
We may therefore assume that~$H$ contains an induced~$D_2^1$.
However, for every $s\geq 2$, $D_s^2$ is $D_2^1$-free and therefore $H$-free, a contradiction.
This completes the proof of~\ref{t-vc:i}.

\medskip
\noindent
We now prove {\sffamily\bfseries\upshape\mathversion{bold}{\ref{t-vc:ii}}}.
Let~$H$ be a graph.

\subparagraph*{\ref{t-vc:ii}: ``$\Leftarrow$''.}
First suppose that~$H$ is an induced subgraph of~$K_{1,3}^+$ or of~$2P_1+\nobreak P_3$.

\subparagraph*{Case 1. $H=K_{1,3}^+$.\\}
Let~$G$ be a $K_{1,3}^+$-free bipartite graph.
We may assume without loss of generality that~$G$ is connected.
By Lemma~\ref{l-alek}, $G$ is either a path, a cycle or an almost complete bipartite graph.
For the first two cases it is readily seen that $\ivc(G)=\vc(G)$.
For the third case we apply Lemma~\ref{almost-complete-bipartite}.

\subparagraph*{Case 2. $H=2P_1+\nobreak P_3$.\\}
Let~$G$ be a $(2P_1+\nobreak P_3)$-free bipartite graph with bipartite classes~$A$ and~$B$, and let~$S$ be a minimum vertex cover of~$G$.
Suppose~$S$ is not an independent set.
Then~$S$ contains two adjacent vertices~$x$ and~$y$, say $x\in A$ and $y\in B$.
Let~$I_x$ and~$I_y$ be the set of neighbours of~$x$ and~$y$, respectively, in $V(G)\setminus S$.
As~$S$ has minimum size, $I_x$ and~$I_y$ are both nonempty.
Moreover, as~$G$ is bipartite, $I_x\cap I_y=\emptyset$.
As the vertices of $G-S$ form an independent set, no two vertices in $I_x\cup I_y$ are adjacent.
Then $|I_x|\leq 1$ or $|I_y|\leq 1$, say $|I_x|\leq 1$, as otherwise~$x$, two vertices of~$I_x$ and two vertices of~$I_y$ form an induced~$2P_1+\nobreak P_3$ in~$G$, a contradiction.

Let $I_x=\{u\}$.
If $|I_y|\geq 2$, we replace~$S$ by $S'=(S\setminus \{x\}) \cup \{u\}$ to obtain another minimum vertex cover of~$G$.
Moreover, $u$ has no neighbours in~$S'$.
In order to see this, let~$z$ be a neighbour of~$u$ in~$S'$, and let~$v_1$,~$v_2$ be two vertices in~$I_y$.
As $V(G)\setminus S$ is an independent set, $u$ is non-adjacent to~$v_1$ and~$v_2$.
As $v_1,v_2,x,z$ all belong to~$A$, they are also mutually non-adjacent.
Hence, the set $\{v_1,v_2,x,u,z\}$ induces a~$2P_1+\nobreak P_3$ in~$G$, a contradiction.
We conclude that replacing~$x$ by~$u$ yields a minimum vertex cover~$S'$ such that~$G[S']$ contains at least one fewer edge than~$G[S]$.

Let now~$S^*$ be a minimum vertex cover such that~$G[S^*]$ has as few edges as possible.
If~$S^*$ is independent, then we have proven that $\ivc(G) = \vc(G)$.
Suppose~$S^*$ is not an independent set.
Then~$S^*$ contains two adjacent vertices~$x^*$ and~$y^*$, say $x^*\in A$ and $y^*\in B$.
By the choice of~$S^*$ and the above discussion, we conclude that each of~$x^*$ and~$y^*$ has exactly one (private) neighbour in $V(G)\setminus S^*$.
Since~$G$ is $(2P_1+\nobreak P_3)$-free, this means that $G-S^*$ has at most three vertices.
The latter implies that at least one of $|A\setminus S^*|$ and $|B\setminus S^*|$, say $|A\setminus S^*|$, has at most one vertex.
But now, since $|B\cap S^*| \geq 1$, it follows that $\ivc(G) \geq \vc(G) = |S^*| = |A \cap S^*| + |B \cap S^*| \geq |A \cap S^*| + |A\setminus S^*| = |A| \geq\ivc(G)$, and hence $\ivc(G) = \vc(G)$.

\subparagraph*{\ref{t-vc:ii}: ``$\Rightarrow$''.}
Now suppose that~$H$ is not an induced subgraph of~$K_{1,3}^+$ or of~$2P_1+\nobreak P_3$.
By~\ref{t-vc:i}, we need only consider the case when~$H$ is an induced subgraph of $K_{1,r}+\nobreak rP_1$ or~$K_{1,r}^+$ for some~$r\geq 1$.
Hence, $H$ contains an induced subgraph from the set $\{K_{1,4},K_{1,3}+\nobreak P_1,\allowbreak 3P_1+\nobreak P_2,5P_1\}$.
Let~$G$ be the double star~$D_{2,2}$ with two leaves for each central vertex, that is, $G$ is the tree on vertices $x,y,u_1,u_2,v_1,v_2$ and edges $xy$, $u_1x$, $u_2x$, $v_1y$, $v_2y$.
We note that~$G$ is bipartite and $(K_{1,4},K_{1,3}+\nobreak P_1,3P_1+\nobreak P_2,5P_1)$-free and thus $H$-free, while $\vc(G)=2$ and $\ivc(G)=3$.
This completes the proof of~\ref{t-vc:ii}.

\medskip
\noindent
We now consider Statements~\ref{t-vc:1}--\ref{t-vc:8}.
Statement~\ref{t-vc:1} follows directly from~\ref{t-vc:ii}, whereas Lemmas~\ref{lem:vc-K1r+sP1} and~\ref{lem:vc-K1r+} imply statements~\ref{t-vc:7} and~\ref{t-vc:8}, respectively.
We prove Statements~\ref{t-vc:2}--\ref{t-vc:6} separately.

\subparagraph*{\ref{t-vc:2}.}
Let~$G$ be a $(K_{1,3}+\nobreak P_1)$-free bipartite graph with partition classes~$A$ and~$B$.
If~$G$ has maximum degree~$2$, then~$G$ is the disjoint union of paths and even cycles, implying that $\ivc(G)=\vc(G)$.
Hence, we may assume that~$G$ contains a vertex~$u$ of degree at least~$3$, say $u\in A$.
Note that~$G$ must be connected, as otherwise~$u$, three neighbours of~$u$ and a vertex from another connected component of~$G$ induce a~$K_{1,3}+\nobreak P_1$ in~$G$.
By the $(K_{1,3}+\nobreak P_1)$-freeness of~$G$, we also find that~$u$ is adjacent to every vertex of~$B$.

First suppose that $|B|\geq 5$.
Consider an arbitrary vertex $u'\in A\setminus \{u\}$.
We find that~$u'$ is adjacent to all but at most two vertices of~$B$, as otherwise~$u$, $u'$ and three non-neighbours of~$u'$ in~$B$ induce a~$K_{1,3}+\nobreak P_1$ in~$G$, a contradiction.
As $|B|\geq 5$, this means that~$u'$ has at least three neighbours in~$B$.
Again by $(K_{1,3}+\nobreak P_1)$-freeness, we find that~$u'$ is also adjacent to all vertices of~$B$.
As~$u'$ is an arbitrary vertex, we conclude that~$G$ is a complete bipartite graph, which implies that $\ivc(G)=\vc(G)$.

Now suppose that $|B|\leq 4$.
As~$B$ is an independent vertex cover of~$G$, we find that $\ivc(G)\leq 4$.
If $\vc(G)=3$, then $\ivc(G) \leq \vc(G) + 1$ (so Statement~\ref{t-vc:2} holds).
If $\vc(G)\leq 1$, then $\ivc(G)=\vc(G)$.
Hence, we may assume that $\vc(G)=2$.
Let $S=\{x,y\}$ be a minimum vertex cover.
If~$S$ is independent, then $\ivc(G)=\vc(G)=2$, so we may assume that~$x$ and~$y$ are adjacent.
As~$G$ is connected, bipartite, and $V(G)\setminus S$ is an independent set, we find that~$G$ is a double star.
As~$G$ is $(K_{1,3}+\nobreak P_1)$-free and contains a vertex of degree at least~$3$, and moreover~$S$ is a minimum vertex cover of~$G$, we find that $G=D_{1,2}$ or $G=D_{2,2}$.
Then $\ivc(G)=\vc(G)$ holds in the former case and $\ivc(G)=\vc(G)+1$ holds in the latter case.
Hence we have proven the bound of~\ref{t-vc:2} and  also, as demonstrated by the graph~$D_{2,2}$, that this bound is tight.

\subparagraph*{\ref{t-vc:3}.}
For some $s \geq 5$, let~$G$ be an $sP_1$-free bipartite graph with partition classes~$A$ and~$B$.
If $\vc(G)\leq 1$, then $\ivc(G)=\vc(G)$ and thus $\ivc(G)\leq \vc(G)+s-3$.
Suppose that $\vc(G)\geq 2$.
As~$G$ is $sP_1$-free, $|A|\leq s-1$ holds.
As~$A$ is an independent vertex cover, this means that $\ivc(G)\leq s-1= 2 + s-3 \leq \vc(G)+s-3$.

\subparagraph*{\ref{t-vc:4b} and~\ref{t-vc:4}.}
Note that the bound for~\ref{t-vc:4} immediately implies~\ref{t-vc:4b}, so it is sufficient to prove Statement~\ref{t-vc:4}.
For some $s\geq 3$, let~$G$ be a $(sP_1+\nobreak P_3)$-free bipartite graph with partition classes~$A$ and~$B$.
Let~$S$ be a minimum vertex cover of~$G$.
First suppose that each vertex of~$S$ has at most one neighbour in $V(G)\setminus S$.
As~$S$ has minimum size, this means that each vertex of~$S$ has exactly one neighbour in $V(G)\setminus S$.
We replace every $u\in S\cap A$ with its unique neighbour in $V(G)\setminus S$, and note that his neighbour belongs to~$B$.
This results in a vertex cover~$S^*$ of the same size as~$S$, but which is a subset of~$B$.
This implies that~$S^*$ is independent.
Thus in this case it follows that $\ivc(G)=\vc(G)$.

Now suppose that~$S$ contains a vertex~$u$, say $u\in A$, with at least two neighbours in $V(G) \setminus S$.
As~$G$ is $(sP_1+\nobreak P_3)$-free and $V(G) \setminus S$ is independent, this means that at most $s-1$ vertices of $G-S$ belong to~$A$.
First suppose that $S\subseteq A$.
Then, as~$A$ is an independent set, we find that~$S$ is independent and thus $\ivc(G)=\vc(G)$.
Now suppose that $S\setminus A\neq \emptyset$, so $|A\cap S|\leq |S|-1$.
As~$A$ is an independent vertex cover of~$G$, we find that $\ivc(G)\leq |A|=|A\cap S| + |A\cap V(G-S)|\leq |S|-1+s-1=\vc(G)+s-2$.

The above bound is tight, as demonstrated by the graph~$D_{s-1,s-1}$, which is $(sP_1+\nobreak P_3)$-free and has $\vc(D_{s-1,s-1})=2$, whereas $\ivc(D_{s-1,s-1})=s-1+1=s=\vc(D_{s-1,s-1})+s-2$.

\subparagraph*{\ref{t-vc:5}.}
For $s \geq 2$, let~$G$ be a $(K_{1,3}+\nobreak sP_1)$-free bipartite graph with partition classes~$A$ and~$B$.
If~$A$ or~$B$ has fewer than $\max\{3s+2, \vc(G)+s\}$ vertices, then we can take the smallest partition class as an independent vertex cover to obtain the desired bound.
We may therefore assume that both~$A$ and~$B$ have size at least $\max\{3s+2, \vc(G)+s\}$.

If every vertex in~$G$ has degree at most~$2$, then~$G$ is $K_{1,3}$-free and by~\ref{t-vc:1} we find that $\ivc(G) = \vc(G)$.
By Lemma~\ref{lem:star+indep-free-bip-structure}, we may therefore assume that fewer than~$s$ vertices of~$A$ have more than $s-1$ non-neighbours in~$B$.
We will show that this leads to a contradiction.

Let~$S$ be a minimum vertex cover of~$G$.
Since~$A$ and~$B$ each have at least $\vc(G)+s$ vertices, there must be a set~$S'$ of $\vc(G)+\nobreak 1$ vertices in~$A$ that has at least  $\vc(G)+\nobreak 1$ neighbours in~$B$.
If a vertex $x \in V(G)$ has degree at least $\vc(G)+\nobreak 1$, then $|N(x)|>|S|$, so $x \in S$.
Therefore every vertex of~$S'$ must be in~$S$, contradicting the fact that $|S'| =\vc(G)+\nobreak 1>\vc(G)=|S|$.

\subparagraph*{\ref{t-vc:6}.}
For some $r\geq 4$, let~$G$ be a $K_{1,r}$-free bipartite graph with partition classes~$A$ and~$B$.
Let~$S$ be a minimum vertex cover of~$G$.
If~$S$ is independent, then $\ivc(G)=\vc(G)$.
Suppose that~$S$ is not independent.
Let $A^*\subseteq A$ be the set of neighbours of the vertices in $S\cap B$.
Note that $|(S\cap A)\cap A^*|\geq 1$, as~$S$ is not independent.
Also note that $(S\cap A)\cup A^*$ is an independent vertex cover of~$G$.
Hence $\ivc(G)\leq |(S\cap A)\cup A^*|=|S\cap A|+|A^*|-|(S\cap A)\cap A^*| \leq |S\cap A|+(r-1)|S\cap B|-1$.
Similarly, $\ivc(G) \leq |S\cap B|+(r-1)|S\cap A|-1$.
Therefore $\ivc(G) \leq \frac{1}{2}(|S\cap A|+(r-1)|S\cap B|-1+|S\cap B|+(r-1)|S\cap A|-1) = \frac{1}{2}(r|S \cap A| + r|S \cap B| -2) = \frac{1}{2}(r|S|)-1=\frac{r}{2}|S|-1$.
To see that this is tight, note that~$D_{r-2,r-2}$ is a $K_{1,r}$-free bipartite graph with $\vc(D_{r-2,r-2})=2$ and $\ivc(D_{r-2,r-2})=r-1=\frac{r}{2}\vc(D_{r-2,r-2})-1$.
\qedllncs
\end{proof}

\section{Feedback Vertex Set}\label{sec:fvs}

In this section we prove Theorems~\ref{thm:indep-fvs} an~\ref{t-fvsi} as part of a more general theorem.
Recall that a graph has an independent feedback vertex set if and only if it is near-bipartite.
We first show the following lemma.

\begin{lemma}\label{lem:fvs-P_1+P_2}
If~$G$ is a $(P_1+\nobreak P_2)$-free near-bipartite graph, then $\ifvs(G) = \fvs(G)$.
\end{lemma}

\begin{proof}
Let~$G$ be a $(P_1+\nobreak P_2)$-free near-bipartite graph.
Note that~$\overline{G}$ is a $P_3$-free graph, so~$\overline{G}$ is a disjoint union of cliques.
It follows that~$G$ is a complete multi-partite graph, say with a partition of its vertex sets into~$k$ non-empty independent sets $V_1,\ldots,V_k$.
We may assume that $k \geq 2$, otherwise~$G$ is an edgeless graph, in which case $\ifvs(G)=\fvs(G)=0$ and we are done.
Since~$G$ is near-bipartite, it contains an independent set~$I$ such that $G-I$ is a forest.
Note that $I \subseteq V_i$ for some $i \in \{1,\ldots,k\}$.
Since near-bipartite graphs are $3$-colourable, it follows that $k \leq 3$.
Furthermore, if $k=3$, then $|V_j|=1$ for some $j \in \{1,2,3\} \setminus \{i\}$, otherwise $G-I$ would contain an induced~$C_4$, a contradiction.
In other words~$G$ is either a complete bipartite graph or the graph formed from a complete bipartite graph by adding a dominating vertex.

First suppose that~$k=2$, so~$G$ is a complete bipartite graph.
Without loss of generality assume that $|V_1| \geq |V_2| \geq 1$.
Let~$S$ be a feedback vertex set of~$G$.
If there are two vertices in $V_1 \setminus S$ and two vertices in $V_2 \setminus S$, then these vertices would induce a~$C_4$ in $G-S$, a contradiction.
Therefore~$S$ must contain all but at most one vertex of~$V_1$ or all but at most one vertex of~$V_2$, so $\fvs(G) \geq \min\{|V_1|-1,|V_2|-1\}=|V_2|-1$.
Let~$I$ be a set consisting of $|V_2|-1$ vertices of~$V_2$.
Then~$I$ is independent and $G-I$ is a star, so~$I$ is an independent feedback vertex set.
It follows that $\ifvs(G) \leq |V_2|-1$.
Since $\fvs(G) \leq \ifvs(G)$, we conclude that $\ifvs(G)=\fvs(G)$ in this case.

Now suppose that~$k=3$, so~$G$ is obtained from a complete bipartite graph by adding a dominating vertex.
Without loss of generality assume that $|V_1| \geq |V_2| \geq |V_3|=1$.
Let~$S$ be a feedback vertex set of~$G$.
By the same argument as in the $k=2$ case, $S$ must contain all but at most one vertex of~$V_1$ or all but at most one vertex of~$V_2$.
If there is a vertex in $V_i \setminus S$ for all $i \in \{1,2,3\}$, then these three vertices would induce a~$C_3$ in $G-S$, a contradiction.
Therefore~$S$ must contain every vertex in~$V_i$ for some $i \in \{1,2,3\}$.
Since $|V_1| \geq |V_2| \geq |V_3|=1$, it follows that $|S| \geq \min\{|V_2|-1+|V_3|,|V_2|\} = |V_2|$.
Therefore $\fvs(G) \geq |V_2|$.
Now~$V_2$ is an independent set and $G-V_2$ is a star, so~$V_2$ is an independent feedback vertex set.
It follows that $\ifvs(G) \leq |V_2|$.
Since $\fvs(G) \leq \ifvs(G)$, we conclude that $\ifvs(G)=\fvs(G)$.\qedllncs
\end{proof}

\begin{lemma}\label{lem:fvs-K1r}
If~$r \geq 1$ and~$G$ is a $K_{1,r}$-free near-bipartite graph, then $\ifvs(G) \leq (2r^2-5r+3) \fvs(G)$.
\end{lemma}

\begin{proof}
Fix integers $k \geq 0$ and $r \geq 1$.
Suppose~$G$ is a $K_{1,r}$-free near-bipartite graph with a feedback vertex set~$S$ such that $|S|=k$.
Since~$G$ is near-bipartite, $V(G)$ can be partitioned into an independent set~$V_1$ and a set $V(G) \setminus V_1$ that induces a forest in~$G$.
Since forests are bipartite, we can partition $V(G) \setminus V_1$ into two independent sets~$V_2$ and~$V_3$.

Suppose $x \in V_i$ for some $i \in \{1,2,3\}$.
Then~$x$ has no neighbours in~$V_i$ since~$V_i$ is an independent set.
For $j \in \{1,2,3\}\setminus \{i\}$, the vertex~$x$ can have at most $r-\nobreak 1$ neighbours in~$V_j$, otherwise~$G$ would contain an induced~$K_{1,r}$.
It follows that $\deg(x) \leq 2(r-1)$ for all $x \in V(G)$.

Let $S'=S$.
Let $F'=V(G) \setminus S'$, so~$G[F']$ is a forest.
To prove the lemma, we will iteratively modify~$S'$ until we obtain an independent feedback vertex set~$S'$ of~$G$ with $|S'| \leq (2r^2-5r+3)|S|$.
Every vertex $u \in S'$ has at most $2r-2$ neighbours in~$F'$.
Consider two neighbours $v,w$ of~$u$ in~$F'$.
As~$F'$ is a forest, there is at most one induced path in~$F'$ from~$v$ to~$w$, so there is at most one induced cycle in $G[F' \cup \{u\}]$ that contains all of~$u$, $v$ and~$w$.
Therefore $G[F' \cup \{u\}]$ contains at most $\binom{2r-2}{2}=\frac{1}{2}(2r-2)(2r-2-1)=2r^2-5r+3$ induced cycles.
Note that every cycle in~$G$ contains at least one vertex of~$V_1$.
Therefore, if $s \in S' \cap (V_2 \cup V_3)$, then we can find a set~$X$ of at most $2r^2-5r+3$ vertices in $V_1 \setminus S'$ such that if we replace~$s$ in~$S'$ by the vertices of~$X$, then we again obtain a feedback vertex set.
Repeating this process iteratively, for each vertex we remove from $S' \cap (V_2 \cup V_3)$, we add at most $2r^2-5r+3$ vertices to $S' \cap V_1$.
We stop the procedure once $S' \cap (V_2 \cup V_3)$ becomes empty, at which point we have produced a feedback vertex set~$S'$ with $|S'| \leq (2r^2-5r+3)|S|$.
Furthermore, at this point $S' \subseteq V_1$, so~$S'$ is independent.
It follows that $\ifvs(G) \leq (2r^2-5r+3)\fvs(G)$.\qedllncs
\end{proof}

Note that all near-bipartite graphs are $3$-colourable (use one colour for the independent set and the two other colours for the forest).
We prove the following lemma.

\begin{lemma}\label{lem:Cr-fvs-and-oct}
Let $k \geq 3$.
The class of $C_k$-free near-bipartite graphs is $\ifvs$-unbounded and $\ioct$-unbounded.
\end{lemma}

\begin{proof}
For $r,s \geq 2$, let~$S_s^r$ denote the graph constructed as follows (see also \figurename~\ref{fig:C_r-fvs-and-oct}).
Start with the graph that is the disjoint union of~$2s$ copies of~$P_{2r}$, and label these copies $U^1,\ldots,U^s,V^1,\ldots,V^s$.
Add a vertex~$u$ adjacent to both endpoints of every~$U^i$ and a vertex~$v$ adjacent to both endpoints of every~$V^i$.
Finally, add an edge between~$u$ and~$v$.

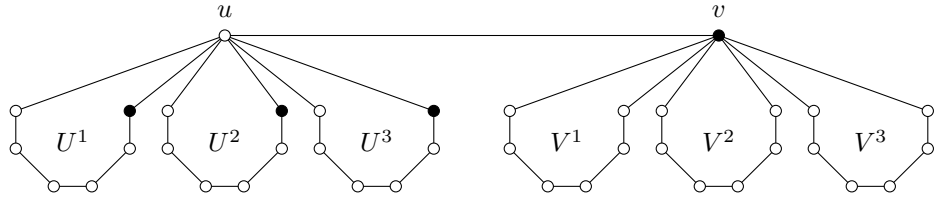
\begin{figure}
\begin{center}
\begin{tikzpicture}[scale=0.5]
\draw (-6.5,1) -- (6.5,1) (-6.5,1) -- (-1,-1) (-6.5,1) -- (-4,-1) (-6.5,1) -- (-5,-1) (-6.5,1) -- (-8,-1) (-6.5,1) -- (-9,-1) (-6.5,1) -- (-12,-1) (6.5,1) -- (1,-1) (6.5,1) -- (4,-1) (6.5,1) -- (5,-1) (6.5,1) -- (8,-1) (6.5,1) -- (9,-1) (6.5,1) -- (12,-1) (1,-1) -- (1,-2) -- (2,-3) -- (3,-3) -- (4,-2) -- (4,-1) (5,-1) -- (5,-2) -- (6,-3) -- (7,-3) -- (8,-2) -- (8,-1) (9,-1) -- (9,-2) -- (10,-3) -- (11,-3) -- (12,-2) -- (12,-1) (-1,-1) -- (-1,-2) -- (-2,-3) -- (-3,-3) -- (-4,-2) -- (-4,-1) (-5,-1) -- (-5,-2) -- (-6,-3) -- (-7,-3) -- (-8,-2) -- (-8,-1) (-9,-1) -- (-9,-2) -- (-10,-3) -- (-11,-3) -- (-12,-2) -- (-12,-1);
\draw [fill=white] (1,-1) circle [radius=4.2pt] (4,-1) circle [radius=4.2pt] (1,-2) circle [radius=4.2pt] (4,-2) circle [radius=4.2pt] (2,-3) circle [radius=4.2pt] (3,-3) circle [radius=4.2pt] (5,-1) circle [radius=4.2pt] (8,-1) circle [radius=4.2pt] (5,-2) circle [radius=4.2pt] (8,-2) circle [radius=4.2pt] (6,-3) circle [radius=4.2pt] (7,-3) circle [radius=4.2pt] (9,-1) circle [radius=4.2pt] (12,-1) circle [radius=4.2pt] (9,-2) circle [radius=4.2pt] (12,-2) circle [radius=4.2pt] (10,-3) circle [radius=4.2pt] (11,-3) circle [radius=4.2pt] (-4,-1) circle [radius=4.2pt] (-1,-2) circle [radius=4.2pt] (-4,-2) circle [radius=4.2pt] (-2,-3) circle [radius=4.2pt] (-3,-3) circle [radius=4.2pt] (-8,-1) circle [radius=4.2pt] (-5,-2) circle [radius=4.2pt] (-8,-2) circle [radius=4.2pt] (-6,-3) circle [radius=4.2pt] (-7,-3) circle [radius=4.2pt] (-12,-1) circle [radius=4.2pt] (-9,-2) circle [radius=4.2pt] (-12,-2) circle [radius=4.2pt] (-11,-3) circle [radius=4.2pt] (-10,-3) circle [radius=4.2pt] (-6.5,1) circle [radius=4.2pt];
\draw [fill=black] (6.5,1) circle [radius=4.2pt] (-1,-1) circle [radius=4.2pt] (-5,-1) circle [radius=4.2pt] (-6,-1) (-9,-1) circle [radius=4.2pt];
\node[above] at (-6.5,1.25) {$u$};
\node[above] at (6.5,1.25) {$v$};
\node[below] at (-10.5,-1.25) {$U^1$};
\node[below] at (-6.5,-1.25) {$U^2$};
\node[below] at (-2.5,-1.25) {$U^3$};
\node[below] at (2.5,-1.25) {$V^1$};
\node[below] at (6.5,-1.25) {$V^2$};
\node[below] at (10.5,-1.25) {$V^3$};
\end{tikzpicture}
\end{center}
\caption{The graph~$S_3^3$.}
\label{fig:C_r-fvs-and-oct}
\end{figure}

Every induced cycle in~$S_s^r$ is isomorphic to~$C_{2r+1}$, which is an odd cycle.
Thus a set $S\subseteq V(S_s^r)$ is a feedback vertex set for~$S_s^r$ if and only if it is an odd cycle transversal for~$S_s^r$.
It follows that $\fvs(S_s^r)=\oct(S_s^r)$ and $\ifvs(S_s^r)=\ioct(S_s^r)$.

Now~$\{u,v\}$ is a minimum feedback vertex set of~$S_s^r$, so $\fvs(S_s^r)=\oct(S_s^r)=2$.
However, any independent feedback vertex set~$S$ contains at most one vertex of~$u$ and~$v$; say it does not contain~$u$.
Then it must contain at least one vertex of each~$U^i$.
It follows that $\ifvs(S_s^r)=\ioct(S_s^r) \geq s+1$.
Since for every  $s \geq 2$, $k \geq 3$, the graph~$S_s^k$ is $C_k$-free, this completes the proof.\qedllncs
\end{proof}

We are now ready to prove the main result of this section, which immediately implies Theorems~\ref{thm:indep-fvs} and~\ref{t-fvsi}.
If an upper bound given in this theorem is tight, that is, if there exists an $H$-free near-bipartite graph~$G$ for which equality holds, we again indicate this by a~$*$ in the corresponding row (whereas the other upper bounds are not known to be tight).

\begin{theorem}\label{t-fvs}
Let~$H$ be a graph.
Then the following two statements hold:
\begin{enumerate}[(i)]
\renewcommand{\theenumi}{(\roman{enumi})}
\renewcommand\labelenumi{(\roman{enumi})}
\item\label{t-fvs:i} the class of $H$-free near-bipartite graphs is $\ifvs$-bounded if and only if~$H$ is isomorphic to~$P_1+\nobreak P_2$, a star or an edgeless graph.
\item\label{t-fvs:ii} for $H\neq K_{1,3}$, the class of $H$-free near-bipartite graphs is $\ifvs$-identical if and only if~$H$ is a (not necessarily induced) subgraph of~$P_3$.\\[-8pt]
\end{enumerate}
In particular, the following statements hold for every $H$-free near-bipartite graph~$G$: 
\begin{enumerate}[(1)*]
\renewcommand{\theenumi}{(\arabic{enumi})}
\renewcommand\labelenumi{(\arabic{enumi})\phantom{*}}
\refstepcounter{enumi}
\item[\textcolor{darkgray}{\sffamily\bfseries\upshape\mathversion{bold}{(1)*}}]\label{t-fvs:1} $\ifvs(G)=\fvs(G)$ if $H\subseteq P_3$
\refstepcounter{enumi}
\item[\textcolor{darkgray}{\sffamily\bfseries\upshape\mathversion{bold}{(2)*}}]\label{t-fvs:2} $\ifvs(G)\leq \fvs(G)+1$ if $H=4P_1$
\item\label{t-fvs:3} $\ifvs(G)\leq\fvs(G)+s-3$ if $H=sP_1$ for $s\geq 5$
\item\label{t-fvs:4} $\ifvs(G)\leq (2r^2-5r+3)\fvs(G)$ if $H=K_{1,r}$ for $r\geq 3$.
\end{enumerate}
\end{theorem}

\begin{proof}
We start by proving~{\sffamily\bfseries\upshape\mathversion{bold}{\ref{t-fvs:i}}}.

\subparagraph*{\ref{t-fvs:i}: ``$\Leftarrow$''.}
First suppose that~$H$ is isomorphic to~$P_1+\nobreak P_2$, a star or an edgeless graph.
If $H=P_1+\nobreak P_2$, then the class of $H$-free near-bipartite graphs is $\ifvs$-bounded by Lemma~\ref{lem:fvs-P_1+P_2}.
If~$H$ is isomorphic to a star or an edgeless graph, then~$H$ is an induced subgraph of~$K_{1,r}$ for some $r \geq 1$.
In this case the class of $H$-free near-bipartite graphs is $\ifvs$-bounded by Lemma~\ref{lem:fvs-K1r}.

\subparagraph*{\ref{t-fvs:i}: ``$\Rightarrow$''.}
Now suppose that the class of $H$-free near-bipartite graphs is $\ifvs$-bounded.
By Lemma~\ref{lem:Cr-fvs-and-oct}, $H$ must be a forest.
We will show that~$H$ is isomorphic to~$P_1+\nobreak P_2$, a star or an edgeless graph.

We start by showing that~$H$ must be $(P_1+\nobreak P_3,2P_1+\nobreak P_2,2P_2)$-free.
Let vertices $x_1,x_2,x_3,x_4$, in that order, form a path on four vertices.
For $s \geq 3$, let~$T_s$ be the graph obtained from this path by adding an independent set~$I$ on~$s$ vertices (see also \figurename~\ref{fig:T_5-for-fvs}) that is complete to the path and note that~$T_s$ is near-bipartite.
Then $\{x_1,x_2,x_3\}$ is a minimum feedback vertex set in~$T_s$.
However, if~$S$ is an independent feedback vertex set, then~$S$ contains at most two vertices in $\{x_1,x_2,x_3,x_4\}$.
Therefore~$S$ must contain at least $s-1$ vertices of~$I$, otherwise $T_s-S$ would contain an induced~$C_3$ or~$C_4$.
Therefore $\fvs(T_s)=3$ and $\ifvs(T_s) \geq s-1$.
Note that~$T_s$ is $(P_1+\nobreak P_3,2P_1+\nobreak P_2,2P_2)$-free (this is easy to see by casting to the complement and observing that~$\overline{T_s}$ is the disjoint union of a~$P_4$ and a complete graph).
Therefore~$H$ cannot contain $P_1+\nobreak P_3$, $2P_1+\nobreak P_2$ or~$2P_2$ as an induced subgraph, otherwise~$T_s$ would be $H$-free, a contradiction.

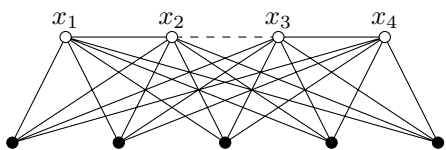
\begin{figure}
\begin{center}
\begin{tikzpicture}[scale=0.7,rotate=270]
\draw (1,4) -- (-1,3) (1,4) -- (-1,1) (1,4) -- (-1,-1) (1,4) -- (-1,-3)
(1,2) -- (-1,3) (1,2) -- (-1,1) (1,2) -- (-1,-1) (1,2) -- (-1,-3)
(1,0) -- (-1,3) (1,0) -- (-1,1) (1,0) -- (-1,-1) (1,0) -- (-1,-3)
(1,-2) -- (-1,3) (1,-2) -- (-1,1) (1,-2) -- (-1,-1) (1,-2) -- (-1,-3)
(1,-4) -- (-1,3) (1,-4) -- (-1,1) (1,-4) -- (-1,-1) (1,-4) -- (-1,-3)
(-1,3) -- (-1,1) (-1,-1) -- (-1,-3);
\draw [dashed] (-1,1) -- (-1,-1);
\draw [fill=white] (-1,3) circle [radius=3pt]
		(-1,1) circle [radius=3pt] 
		(-1,-1) circle [radius=3pt] 
		(-1,-3) circle [radius=3pt];
\draw [fill=black](1,4) circle [radius=3pt] 
		(1,2) circle [radius=3pt] 
		(1,0) circle [radius=3pt]
		(1,-2) circle [radius=3pt]
		(1,-4) circle [radius=3pt]; 
\node[above] at (-1,3) {$x_4$};
\node[above] at (-1,1) {$x_3$};
\node[above] at (-1,-1) {$x_2$};
\node[above] at (-1,-3) {$x_1$};
\end{tikzpicture}
\end{center}
\caption{\label{fig:T_5-for-fvs}The graphs~$T_5$ and~$T_5'$.
The edge~$x_2x_3$ is present in~$T_5$, but not in~$T_5'$.}
\end{figure}

Next, we show that~$H$ must be $P_4$-free.
For $s \geq 3$ let~$T_s'$ be the graph obtained from~$T_s$ by removing the edge~$x_2x_3$ (see also \figurename~\ref{fig:T_5-for-fvs}).
Then $\{x_1,x_2,x_3\}$ is a minimum feedback vertex set in~$T_s'$, so $\fvs(T_s') =3$.
By the same argument as for~$T_s$, we find that $\ifvs(T_s') \geq s-1$.
Now the complement~$\overline{T_s'}$ is the disjoint union of a~$C_4$ and a complete graph, so~$T_s'$ is $P_4$-free.
Therefore~$H$ cannot contain~$P_4$ as an induced subgraph.

We may now assume that~$H$ is a $(P_1+\nobreak P_3,2P_1+\nobreak P_2,2P_2,P_4)$-free forest.
If~$H$ is connected, then it is a $P_4$-free tree, so it is a star, in which case we are done.
We may therefore assume that~$H$ is disconnected.
We may also assume that~$H$ contains at least one edge, otherwise we are done.
Since~$H$ is $(2P_1+\nobreak P_2)$-free, it cannot have more than two components.
Since~$H$ is $2P_2$-free, one of its two components must be isomorphic to~$P_1$.
Since~$H$ is a $(P_1+\nobreak P_3)$-free forest, its other component must be isomorphic to~$P_2$.
Hence~$H$ is isomorphic to $P_1+\nobreak P_2$.
This completes the proof of~\ref{t-fvs:i}.

\medskip
\noindent
We now prove~{\sffamily\bfseries\upshape\mathversion{bold}{\ref{t-fvs:ii}}}.
Let~$H$ be a graph not isomorphic to~$K_{1,3}$.

\subparagraph*{\ref{t-fvs:ii}: ``$\Leftarrow$''.}
First suppose that~$H$ is a subgraph of~$P_3$.
If $H\subseteq_i P_1+\nobreak P_2$, then $\ifvs(G)=\fvs(G)$ for every $H$-free near-bipartite graph~$G$ by Lemma~\ref{lem:fvs-P_1+P_2}.
If $H\subseteq_i P_3$, then every $H$-free near-bipartite graph~$G$ is a disjoint union of complete graphs on at most three vertices, and hence $\ifvs(G)=\fvs(G)$ holds.
Finally, suppose that $H\subseteq_i 3P_1$.
Let~$G$ be a $3P_1$-free near-bipartite graph.
As~$G$ is $3P_1$-free, every minimum independent feedback vertex set of~$G$ has size at most~$2$.
Hence, every minimum feedback vertex set of~$G$ also has size at most~$2$.
Moreover, if it has size less than~$2$, then it is an independent feedback vertex set.
We conclude that $\ifvs(G)=\fvs(G)$.

\subparagraph*{\ref{t-fvs:ii}: ``$\Rightarrow$''.}
Now suppose that~$H$ is not a subgraph of~$P_3$.
Recall that we assume that $H\neq K_{1,3}$.
By~\ref{t-fvs:i}, we may then assume that $H=K_{1,r}$ for some $r\geq 4$ or $H=sP_1$ for some $s\geq 4$.
Consider the graph~$G$ in \figurename~\ref{fvs1}.
It is straightforward to check that~$G$ is $4P_1$-free and near-bipartite; $\{u,v\}$ is a minimum feedback vertex set (indeed $G-\{u,v\}$ is~$P_5$) while $\ifvs(G)=3$; for instance, $\{v,v_2,v_3\}$ is a minimum independent feedback vertex set of~$G$.
This completes the proof of~\ref{t-fvs:ii}.

\medskip
\noindent
We now consider Statements~\ref{t-fvs:1}--\ref{t-fvs:4}.
Statement~\ref{t-fvs:1} follows directly from Statement~\ref{t-fvs:ii}, whereas Lemma~\ref{lem:fvs-K1r} implies Statement~\ref{t-fvs:4}.
We prove Statements~\ref{t-fvs:2} and~\ref{t-fvs:3} below.

\subparagraph*{\ref{t-fvs:2} and~\ref{t-fvs:3}.}
First note that, as shown in the proof of Statement~\ref{t-fvs:ii}, the graph~$G$ in \figurename~\ref{fvs1} is $4P_1$-free, with $\fvs(G)=2$ and $\ifvs(G)=3$, so the bound in Statement~\ref{t-fvs:2} is tight.
It remains to prove that $\ifvs(G)\leq\fvs(G)+s-3$ if $H=sP_1$ with $s\geq 4$ (this proves the bounds in Statements~\ref{t-fvs:2} $(s=4)$ and~\ref{t-fvs:3} $(s\geq 5)$).
Let~$G$ be an $sP_1$-free near-bipartite graph.
If $\fvs(G)\leq 1$, then $\ifvs(G)=\fvs(G)$.
Hence, we may assume that $\fvs(G)\geq 2$.
As~$G$ is near-bipartite, $V(G)$ can be partitioned into three independent sets $V_1$, $V_2$,~$V_3$, such that $V_2\cup V_3$ induce a forest.
Hence, $V_1$ is an independent feedback vertex set.
As~$G$ is $sP_1$-free, $V_1$ has size at most~$s-1$.
This means that $\ifvs(G)\leq s-1 = 2 + s-3 \leq \fvs(G)+s-3$.
This completes the proof of Statements~\ref{t-fvs:2} and~\ref{t-fvs:3}.\qedllncs
\end{proof}

\begin{figure}
\begin{center}
\begin{tikzpicture}[xscale=0.8, yscale=0.8]
\draw (2,-2) -- (1,-1) -- (0,-1) -- (-2,-2) -- (0,1) -- (-2,2) -- (2,2) -- (2,-2) -- (-2,-2) -- (-2,2) (0,-1) -- (0,1) -- (2,-2);
\draw[fill=black] (0,1) circle [radius=3pt] (2,-2) circle [radius=3pt];
\draw[fill=white] (-2,-2) circle [radius=3pt] (-2,2) circle [radius=3pt] (2,2) circle [radius=3pt] (0,-1) circle [radius=3pt] (1,-1) circle [radius=3pt];
\node[right] at (0,1) {$u$};
\node[right] at (2,-2) {$v$};
\node[above left] at (0,-1) {$v_2$};
\node[left] at (-2,2) {$v_3$};
\end{tikzpicture}
\end{center}
\caption{An example of a $4P_1$-free near-bipartite graph~$G$ with $\ifvs(G)=\fvs(G)+1$, which shows that the bound in Theorem~\ref{t-fvs}\ref{t-fvs:2} is tight.}\label{fvs1}
\end{figure}
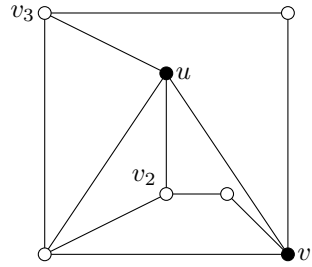

\section{Odd Cycle Transversal}\label{sec:oct}

In this section we prove Theorems~\ref{thm:indep-oct} and~\ref{t-octi} as part of a more general theorem.
Recall that a graph has an independent odd cycle transversal if and only if it is $3$-colourable.
Before proving the main result of this section, we first provide a sequence of auxiliary statements.

\begin{lemma}\label{lem:oct-P4}
If~$G$ is a $P_4$-free $3$-colourable graph, then $\ioct(G)=\oct(G)$.
\end{lemma}

\begin{proof}
Let~$G$ be a $P_4$-free $3$-colourable graph.
It suffices to prove the lemma component-wise, so we may assume that~$G$ is connected.
Note that~$G$ cannot contain any induced odd cycles on more than three vertices, as it is $P_4$-free.
Let $(V_1,V_2,V_3)$ be a partition of~$V(G)$ into independent sets.
We may assume that~$G$ is not bipartite, otherwise $\ioct(G)=\oct(G)=0$, in which case we are done.
As~$G$ is connected, $P_4$-free and contains more than one vertex, its complement~$\overline{G}$ must be disconnected.
Therefore we can partition the vertex set of~$G$ into two parts~$X_1$ and~$X_2$ such that~$X_1$ is complete to~$X_2$.
No independent set~$V_i$ can have vertices in both~$X_1$ and~$X_2$, so without loss of generality we may assume that $X_1=V_1$ and $X_2 = V_2 \cup V_3$.
Since~$G[X_2]$ is a $P_4$-free bipartite graph, it is readily seen that it is a disjoint union of complete bipartite graphs.

Note that $G-X_1$ is a bipartite graph, so~$X_1$ is an odd cycle transversal of~$G$.
Furthermore, $X_1$ is independent.
Now let~$S$ be a minimum vertex cover of~$G[X_2]$.
Observe that $G-S$ is bipartite, so~$S$ is an odd cycle transversal of~$G$.
Since~$G[X_2]$ is the disjoint union of complete bipartite graphs, for every component~$C$ of~$G[X_2]$, $S$ must contain one part of the bipartition of~$C$, or the other; by minimality of~$S$, it only contains vertices from one of the parts.
It follows that~$S$ is independent.

We now claim that every minimum odd cycle transversal~$S$ of~$G$ contains either~$X_1$ or a minimum vertex cover of~$G[X_2]$, both of which we have shown are independent odd cycle transversals; by the minimality of~$S$, this will imply that~$S$ is equal to one of them.
Indeed, suppose for contradiction that~$S$ is a minimum odd cycle transversal such that there is a vertex $x \in X_1 \setminus S$ and two adjacent vertices $y,z \in X_2 \setminus S$.
Then $G[\{x,y,z\}]$ is a~$C_3$ in $G-S$.
This contradiction completes the proof.\qedllncs
\end{proof}

\begin{lemma}\label{lem:oct-K13}
If~$G$ is a $K_{1,3}$-free $3$-colourable graph, then $\ioct(G) \leq 3\oct(G)$.
\end{lemma}
\begin{proof}
Fix an integer $k \geq 0$.
Let~$G$ be a $K_{1,3}$-free $3$-colourable graph with an odd cycle transversal~$S$ such that $|S|=k$.
Fix a partition of~$V(G)$ into three independent sets $V_1,V_2,V_3$.
Without loss of generality assume that $|S \cap V_1| \geq |S \cap V_2|, |S\cap V_3|$, so $|S \cap (V_2 \cup V_3)| \leq \frac{2k}{3}$.
Let $S'=S$ and note that $G-S'$ is bipartite by definition of odd cycle transversal.
To prove the lemma, we will iteratively modify~$S'$ until we obtain an independent odd cycle transversal~$S'$ of~$G$ with $|S'| \leq 3k$.

Suppose $x \in V_i$ for some $i \in \{1,2,3\}$.
Then~$x$ has no neighbours in~$V_i$ since~$V_i$ is an independent set.
For $j \in \{1,2,3\}\setminus \{i\}$, the vertex~$x$ can have at most two neighbours in~$V_j$, otherwise~$G$ would contain an induced~$K_{1,3}$.
It follows that $\deg(x) \leq 4$ for all $x \in V(G)$.

As $G-S'$ is a bipartite $K_{1,3}$-free graph, it is a disjoint union of paths and even cycles.
Every vertex $u \in S'$ has at most four neighbours in $V(G) \setminus S'$.
An induced odd cycle in $G-(S' \setminus \{u\})$ consists of the vertex~$u$ and an induced path~$P$ in $G-S'$ between two neighbours $v,w$ of~$u$ such that $P\cap N(u)$ does not contain any vertices apart from~$v$ and~$w$.
If~$u$ has~$q$ neighbours in some component~$C$ of $G-S'$, then there can be at most~$q$ such paths~$P$ that lie in this component.
It follows that there are at most four induced odd cycles in $G-(S' \setminus \{u\})$.
Note that every induced odd cycle in~$G$ contains at least one vertex in each~$V_i$.
Therefore, if $s \in S' \cap (V_2 \cup V_3)$, then we can find a set~$X$ of at most four vertices in $V_1 \setminus S'$ such that if we replace~$s$ in~$S'$ by the vertices of~$X$, then we again obtain an odd cycle transversal.
Repeating this process iteratively, for each vertex we remove from $S' \cap (V_2 \cup V_3)$, we add at most four vertices to $S' \cap V_1$, so~$|S'|$ increases by at most~$3$.
We stop the procedure once $S' \cap (V_2 \cup V_3)$ becomes empty, at which point we have produced an odd cycle transversal~$S'$ with $|S'| \leq |S| + 3|S \cap (V_2 \cup V_3)| \leq k + 3\times \frac{2k}{3}=3k$.
Furthermore, at this point $S' \subseteq V_1$, so~$S'$ is independent.
It follows that $\ioct(G) \leq 3 \oct(G)$.\qedllncs 
\end{proof}

\begin{lemma}\label{lem:oct-K_1r-to-K13+sP_1}
Let $r,s \geq 1$.
Suppose there is a function $f:\mathbb{Z}_{\geq 0} \to \mathbb{Z}_{\geq 0}$ such that $\ioct(G) \leq f(\oct(G))$ for every $K_{1,r}$-free $3$-colourable graph~$G$.
Then $\ioct(G) \leq \max\{\oct(G)r+r^2+3rs-2r,\allowbreak f(\oct(G))\}$ for every $(K_{1,r}+\nobreak sP_1)$-free $3$-colourable graph~$G$.
\end{lemma}

\begin{proof}
Fix $r,s \geq 1$ and $k \geq 0$.
Let~$G$ be a $(K_{1,r}+\nobreak sP_1)$-free $3$-colourable graph with a minimum odd-cycle transversal~$T$ on~$k$ vertices.
Fix a partition of~$V(G)$ into three independent sets $V_1,V_2,V_3$.
We may assume that $\oct(G) \geq 2$, otherwise $\ioct(G)=\oct(G)$ and we are done.
If $|V_i| \leq \max\{\oct(G)r+r^2+3rs-2r,\allowbreak f(\oct(G))\}$ for some $i \in \{1,2,3\}$, then deleting~$V_i$ from~$G$ yields a bipartite graph, so $\ioct(G) \leq \max\{\oct(G)r+r^2+3rs-2r,\allowbreak f(\oct(G))\}$ and we are done.
We may therefore assume that $|V_i| > \max\{\oct(G)r+r^2+3rs-2r,\allowbreak f(\oct(G))\}$ for all $i \in \{1,2,3\}$.
If~$G$ is $K_{1,r}$-free, then $\ioct(G) \leq f(\oct(G))$, so suppose that~$G$ contains an induced~$K_{1,r}$, say with vertex set~$X$.
Note that $|X|=r+1$, and each~$V_i$ can contain at most~$r$ vertices of~$X$, since every~$V_i$ is an independent set.

For every $i \in \{1,2,3\}$, there cannot be a set of~$s$ vertices in $V_i \setminus X$ that are anti-complete to~$X$, otherwise~$G$ would contain an induced $K_{1,r}+\nobreak sP_1$, a contradiction.
For every $i \in \{1,2,3\}$, since $|V_i| > \oct(G)r+r^2+3rs-2r \geq r^2+3rs$, it follows that $|V_i \setminus X|\geq |V_i|- r>(s-1)+(r+1)(r-1)=(s-1)+|X|(r-1)$.
Hence for every $i \in \{1,2,3\}$, there must be a vertex $x \in X$ that has at least~$r$ neighbours in~$V_i$.
Applying this for each~$i$ in turn, we find that at least two of the graphs in $\{G[V_1 \cup V_2],G[V_1 \cup V_3],G[V_2 \cup V_3]\}$ contain a vertex of degree at least~$r$; without loss of generality assume that this is the case for $G[V_1 \cup V_2]$ and $G[V_1 \cup V_3]$.
Let~$V_2'$ and~$V_3'$ denote the set of vertices in~$V_2$ and~$V_3$, respectively, that have more than $s-\nobreak 1$ non-neighbours in~$V_1$.
By Lemma~\ref{lem:star+indep-free-bip-structure}, $|V_2'|,|V_3'| \leq s-1$.

Suppose a vertex $x \in V_2 \setminus V_2'$ is adjacent to a vertex $y \in V_3 \setminus V_3'$.
By definition of~$V_2'$ and~$V_3'$, the vertices~$x$ and~$y$ each have at most $s-\nobreak 1$ non-neighbours in~$V_1$.
Since $|V_1| - 2(s-1) \geq \oct(G)+1$, it follows that $|N(x) \cap N(y) \cap V_1| \geq \oct(G)+1$ so $N(x) \cap N(y) \cap V_1 \not \subseteq T$.
We conclude that at least one of~$x$ or~$y$ must be in~$T$.
In other words, $T \cap ((V_2 \setminus V_2') \cup (V_3 \setminus V_3'))$ is a vertex cover of $G[(V_2 \setminus V_2') \cup (V_3 \setminus V_3')]$, of size at most~$\oct(G)$.
Therefore $(T \cap ((V_2 \setminus V_2') \cup (V_3 \setminus V_3'))) \cup V_2' \cup V_3'$ is a vertex cover of $G[V_2 \cup V_3]$ of size at most $\oct(G)+2(s-1)$.
By Lemma~\ref{lem:vc-K1r+sP1}, there is an independent vertex cover~$T'$ of $G[V_2 \cup V_3]$ of size at most $(\oct(G)+2(s-1))r+rs=\oct(G)r+3rs-2r$.
Note that by definition of vertex cover, $(V_2 \cup V_3) \setminus T'$ is an independent set, and so $G-T'$ is bipartite.
Therefore~$T'$ is an independent odd cycle transversal for~$G$ of size at most~$\oct(G)r+3rs-2r$.
This completes the proof.\qedllncs
\end{proof}

The following result follows immediately from combining Lemmas~\ref{lem:oct-K13} and~\ref{lem:oct-K_1r-to-K13+sP_1}.

\begin{corollary}\label{cor:oct-K13+sP1}
For $s \geq 1$, $\ioct(G) \leq 3\oct(G)+9s+3$ for every $(K_{1,3}+\nobreak sP_1)$-free $3$-colourable graph~$G$.
\end{corollary}

\begin{lemma}\label{lem:P1+P4-2P2-oct}
The class of $(P_1+\nobreak P_4,2P_2)$-free $3$-colourable graphs is $\ioct$-unbounded.
\end{lemma}

\begin{proof}
Let $s \geq 2$.
We construct the graph~$Q_s$ as follows (see also \figurename~\ref{fig:Qr-oct}).
First, let $A$, $B$ and~$C$ be disjoint independent sets of~$s$ vertices.
Choose vertices $a\in A$, $b\in B$ and $c\in C$.
Add edges so that~$a$ is complete to $B \cup C$, $b$ is complete to $A \cup C$ and~$c$ is complete to $A \cup B$.
Let~$Q_s$ be the resulting graph and note that it is $3$-colourable with colour classes $A$, $B$ and~$C$.

\begin{figure}
\begin{center}
\begin{tikzpicture}[scale=0.7,rotate=-90]
\coordinate (m1) at (0:1.4);
\coordinate (m2) at (120:1.4);
\coordinate (m3) at (240:1.4);

\coordinate (x1) at (0:2);
\coordinate (x2) at (120:2);
\coordinate (x3) at (240:2);

\coordinate (y1) at (15+180+0:4);
\coordinate (y2) at (15+180+120:4);
\coordinate (y3) at (15+180+240:4);

\coordinate (z1) at (-15+180+0:4);
\coordinate (z2) at (-15+180+120:4);
\coordinate (z3) at (-15+180+240:4);

\coordinate (w1) at (180+0:4);
\coordinate (w2) at (180+120:4);
\coordinate (w3) at (180+240:4);

\coordinate (l1) at (180+0:4.8);
\coordinate (l2) at (180+120:5.3);
\coordinate (l3) at (180+240:5.3);

\coordinate (e1) at ($(y1)!0.5!(z1)$);
\coordinate (e2) at ($(y2)!0.5!(z2)$);
\coordinate (e3) at ($(y3)!0.5!(z3)$);

\draw[fill=gray!20!white] (e1) ellipse (0.5 and 1.5);
\draw[fill=gray!20!white,rotate=120] (e2) ellipse (0.5 and 1.5);
\draw[fill=gray!20!white,rotate=240] (e3) ellipse (0.5 and 1.5);

\draw(l1) node {$A \setminus \{a\}$};
\draw(l2) node {$B \setminus \{b\}$};
\draw(l3) node {$C \setminus \{c\}$};

\draw(m1) node {$a$};
\draw(m2) node {$b$};
\draw(m3) node {$c$};

\draw (x1) -- (x2) -- (x3) -- (x1);

\draw (x1) -- (y3) -- (x2);
\draw (x1) -- (z3) -- (x2);
\draw (x1) -- (w3) -- (x2);

\draw (x1) -- (y2) -- (x3);
\draw (x1) -- (z2) -- (x3);
\draw (x1) -- (w2) -- (x3);

\draw (x2) -- (y1) -- (x3);
\draw (x2) -- (z1) -- (x3);
\draw (x2) -- (w1) -- (x3);

\draw[fill=black](x1) circle [radius=3pt];
\draw[fill=white](x2) circle [radius=3pt];
\draw[fill=white](x3) circle [radius=3pt];

\draw[fill=black](y1) circle [radius=3pt];
\draw[fill=white](y2) circle [radius=3pt];
\draw[fill=white](y3) circle [radius=3pt];

\draw[fill=black](z1) circle [radius=3pt];
\draw[fill=white](z2) circle [radius=3pt];
\draw[fill=white](z3) circle [radius=3pt];

\draw[fill=black](w1) circle [radius=3pt];
\draw[fill=white](w2) circle [radius=3pt];
\draw[fill=white](w3) circle [radius=3pt];
\end{tikzpicture}
\end{center}
\caption{The graph~$Q_4$.}
\label{fig:Qr-oct}
\end{figure}

Note that~$\{a,b\}$ is a minimum odd cycle transversal of~$Q_s$, so $\oct(Q_s)=2$.

Let~$S$ be a minimum independent odd cycle transversal.
Then~$S$ contains at most one vertex in $\{a,b,c\}$, say~$S$ contains neither~$b$ nor~$c$.
If a vertex $x \in A$ is not in~$S$, then $Q_s[\{x,b,c\}]$ is a~$C_3$ in $Q_s-S$, a contradiction.
Hence every vertex of~$A$ is in~$S$, and so $\ioct(Q_s) \geq s$.

It remains to show that~$Q_s$ is $(P_1+\nobreak P_4,2P_2)$-free.
Consider a vertex $x \in A$.
Then $Q_s-N[x]$ is an edgeless graph if $x=a$ and $Q_s-N[x]$ is the disjoint union of a star and an edgeless graph otherwise.
It follows that $Q_s-N[x]$ is $P_4$-free.
By symmetry, we conclude that~$Q_s$ is $(P_1+\nobreak P_4)$-free.
Now consider a vertex $y \in N(a) \cap B$.
Then $Q_s-N[\{a,y\}]$ is empty if $y=b$ and $Q_s-N[\{a,y\}]$ is an edgeless graph otherwise.
It follows that $Q_s-N[\{a,y\}]$ is $P_2$-free.
By symmetry, we conclude that~$Q_s$ is $2P_2$-free.
This completes the proof.\qedllncs
\end{proof}

\begin{lemma}\label{lem:oct-two-deg3}
Let~$H$ be a graph with more than one vertex of degree at least~$3$.
Then the class of $H$-free $3$-colourable graphs is $\ioct$-unbounded.
\end{lemma}
\begin{proof}
Let $s \geq 1$.
We construct the graph~$Z_s$ as follows (see also \figurename~\ref{fig:oct-deg3}).
Start with the disjoint union of~$s$ copies of~$P_4$ and label these copies $U^1,\ldots,U^s$.
Add an edge~$ab$ and make~$a$ and~$b$ adjacent to both endpoints of every~$U^i$.
Let~$Z_s$ be the resulting graph and note that~$Z_s$ is $3$-colourable (colour~$a$ and~$b$ with Colours~$1$ and~$2$, respectively, colour the endpoints of the~$U^i$s with Colour~$3$ and colour the remaining vertices of the~$U^i$s with Colours~$1$ and~$2$).

\begin{figure}
\begin{center}
\begin{tikzpicture}[xscale=0.5, yscale=0.5]
\draw (-3,2) -- (-9,-1) -- (-9,-3) -- (-6,-3) -- (-6,-1) -- (-3,2) (-9,-1) -- (3,2) -- (-6,-1) (-3,2) -- (-4,-1) -- (-4,-3) -- (-1,-3) -- (-1,-1) -- (-3,2) (-4,-1) -- (3,2) -- (-1,-1) (-3,2) -- (6,-1) -- (6,-3) -- (9,-3) -- (9,-1) -- (-3,2) (9,-1) -- (3,2) -- (6,-1) (-3,2) -- (1,-1) -- (1,-3) -- (4,-3) -- (4,-1) -- (-3,2) (4,-1) -- (3,2) -- (1,-1) (-3,2) -- (3,2);
\draw[fill=white] (-9,-1) circle [radius=4.2pt];
\draw[fill=white] (-6,-1) circle [radius=4.2pt];
\draw[fill=white] (-4,-1) circle [radius=4.2pt];
\draw[fill=white] (-1,-1) circle [radius=4.2pt];
\draw[fill=white] (6,-1) circle [radius=4.2pt];
\draw[fill=white] (9,-1) circle [radius=4.2pt];
\draw[fill=black] (-9,-3) circle [radius=4.2pt];
\draw[fill=white] (-6,-3) circle [radius=4.2pt];
\draw[fill=black] (-4,-3) circle [radius=4.2pt];
\draw[fill=white] (-1,-3) circle [radius=4.2pt];
\draw[fill=black] (6,-3) circle [radius=4.2pt];
\draw[fill=white] (9,-3) circle [radius=4.2pt];
\draw[fill=white] (-3,2) circle [radius=4.2pt];
\draw[fill=black] (3,2) circle [radius=4.2pt];
\draw[fill=white] (4,-3) circle [radius=4.2pt];
\draw[fill=black] (1,-3) circle [radius=4.2pt];
\draw[fill=white] (4,-1) circle [radius=4.2pt];
\draw[fill=white] (1,-1) circle [radius=4.2pt];
\node [above] at (-3,2.1) {$a$};
\node [above] at (3,2.1) {$b$};
\node [below] at (-7.5,-3.15) {$U^1$};
\node [below] at (-2.5,-3.15) {$U^2$};
\node [below] at (2.5,-3.15) {$U^3$};
\node [below] at (7.5,-3.15) {$U^4$};
\end{tikzpicture}
\end{center}
\caption{The graph~$Z_4$.}
\label{fig:oct-deg3}
\end{figure}

Note that $Z_s-\{a,b\}$ is bipartite, so~$\{a,b\}$ is a minimum odd cycle transversal and $\oct(Z_s)=2$.
However, any independent odd cycle transversal~$S$ contains at most one vertex of~$a$ and~$b$; say it does not contain~$a$.
For every $i \in \{1,\ldots,s\}$, the graph $Z_s[U^i \cup \{a\}]$ is a~$C_5$.
Therefore~$S$ must contain at least one vertex from each~$U^i$.
It follows that $\ioct(Z_s) \geq s$.

Let~$H$ be a graph with more than one vertex of degree at least~$3$.
By Lemma~\ref{lem:Cr-fvs-and-oct}, we may assume that~$H$ is a forest.
It remains to show that~$Z_s$ is $H$-free.
Suppose, for contradiction, that~$Z_s$ contains~$H$ as an induced subgraph and let~$x$ and~$y$ be two vertices that have degree at least~$3$ in~$H$.
Since~$H$ is a forest, $x$ and~$y$ must each have three pairwise non-adjacent neighbours in~$Z_s$.
The endpoints of each~$U^i$ have exactly three neighbours, but two of them ($a$ and~$b$) are adjacent.
Without loss of generality we may therefore assume that $x=a$ and $y=b$.
Since~$x$ has degree at least~$3$ in~$H$, the vertex~$x$ must have a neighbour $z \neq y$ in~$H$ and so~$z$ must be the endpoint of a~$U^i$.
Therefore~$x$, $y$ and~$z$ are pairwise adjacent, so $H[\{x,y,z\}]$ is a~$C_3$, contradicting the fact that~$H$ is a forest.
It follows that~$Z_s$ is $H$-free.
This completes the proof.\qedllncs
\end{proof}

\begin{lemma}\label{lem:oct-K15}
The class of $K_{1,5}$-free $3$-colourable graphs is $\ioct$-unbounded.
\end{lemma}

\begin{proof}
Let $s \geq 1$.
We construct the graph~$Y_s$ as follows (see also \figurename~\ref{fig:oct-K15}).
\begin{enumerate}
\item Start with the disjoint union of four copies of~$P_{3s}$ and label the vertices of these paths $a_1,\ldots,a_{3s}$, $b_1,\ldots,b_{3s}$, $c_1,\ldots,c_{3s}$ and $d_1,\ldots,d_{3s}$ in order, respectively.
\item For each $i \in \{1,\ldots,3s\}$ add the edges~$a_ib_i$ and~$c_id_i$.
\item For each $i \in \{1,\ldots,3s-1\}$ add the edges~$a_ic_{i+1}$ and~$d_ib_{i+1}$.
\item Finally, add an edge~$xy$ and make~$x$ adjacent to~$a_1$ and~$d_1$ and~$y$ adjacent to~$a_1$, $b_1$, $c_1$ and~$d_1$.
\end{enumerate}
Let~$Y_s$ be the resulting graph.

\begin{figure}
\begin{center}
\begin{tikzpicture}
\foreach \letter/\j in {a/1.7,b/0.7,c/-0.7,d/-1.7}
\foreach \i in {1,2,...,6}{
\coordinate (\letter\i) at (\i*1.2,\j*1.2);
}

\coordinate (x) at (-2,0);
\coordinate (y) at (-1,0);
\draw (x) -- (y);

\foreach \i in {1,2,...,6} {\draw (a\i)--(b\i) (c\i) -- (d\i);}
\foreach \i[evaluate=\i as \evali using int(\i+1)] in {1,2,...,5} {\draw (a\evali) -- (a\i)--(c\evali)--(c\i) (d\evali) -- (d\i) -- (b\evali) -- (b\i);}

\foreach \from\to in {x/a1,x/d1,y/a1,y/b1,y/c1,y/d1} \draw (\from) -- (\to);

\tikzstyle{w_ci_vertex}=[circle,fill=white!100,text=black,minimum size=0.3cm,draw]
\tikzstyle{b_ci_vertex}=[circle,fill=black!50,text=black,minimum size=0.3cm,draw]
\tikzstyle{w_sq_vertex}=[regular polygon,regular polygon sides=4,fill=white!100,text=black,scale=0.9,draw]
\tikzstyle{b_sq_vertex}=[regular polygon,regular polygon sides=4,fill=black!50,text=black,scale=0.9,draw]
\tikzstyle{w_tr_vertex}=[regular polygon,regular polygon sides=3,fill=white!100,text=black,scale=0.6,draw]
\tikzstyle{b_tr_vertex}=[regular polygon,regular polygon sides=3,fill=black!50,text=black,scale=0.6,draw]

\node[w_sq_vertex,label=left:{$x$}] at (x) {};
\node[w_tr_vertex,label={[label distance=-3pt]above left:$y$}] at (y) {};

\foreach \letter/\i/\pos/\type in
{
a/1/0.4/b_ci_vertex,
a/2/0.4/w_sq_vertex,
a/3/0.4/b_tr_vertex,
a/4/0.4/w_ci_vertex,
a/5/0.4/b_sq_vertex,
a/6/0.4/w_tr_vertex,
b/1/-0.45/w_sq_vertex,
b/2/-0.45/b_tr_vertex,
b/3/-0.45/w_ci_vertex,
b/4/-0.45/b_sq_vertex,
b/5/-0.45/w_tr_vertex,
b/6/-0.45/b_ci_vertex,
c/1/0.4/b_sq_vertex,
c/2/0.4/w_tr_vertex,
c/3/0.4/b_ci_vertex,
c/4/0.4/w_sq_vertex,
c/5/0.4/b_tr_vertex,
c/6/0.4/w_ci_vertex,
d/1/-0.45/w_ci_vertex,
d/2/-0.45/b_sq_vertex,
d/3/-0.45/w_tr_vertex,
d/4/-0.45/b_ci_vertex,
d/5/-0.45/w_sq_vertex,
d/6/-0.45/b_tr_vertex} 
{
\node[\type] at (\letter\i) {};
\node at ($(\letter\i)+(0,\pos)$) {$\letter_\i$};
}
\end{tikzpicture}
\hspace*{1.5cm}
\end{center}
\caption{The graph~$Y_2$.
Different shapes show the unique $3$-colouring of~$Y_2$.
Different colours show the $2$-colouring of $Y_2 - \{ x, y \}$.}
\label{fig:oct-K15}
\end{figure}
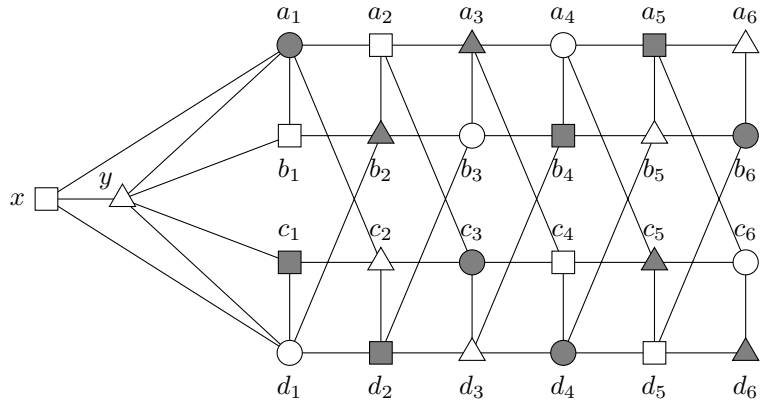

First note that~$Y_s$ is $K_{1,5}$-free.
The vertices~$y$, $a_1$ and~$d_1$ all have degree~$5$, but their neighbourhood is not independent, so they cannot be the central vertex of an induced~$K_{1,5}$.
All the other vertices have degree at most~$4$, so they cannot be the central vertex of an induced~$K_{1,5}$ either.
Therefore no vertex in~$Y_s$ is the central vertex of an induced~$K_{1,5}$, so~$Y_s$ is $K_{1,5}$-free.

The graph $Y_s-\{x,y\}$ is bipartite with bipartition classes:
\begin{enumerate}
\item $\{a_i, c_i \;|\; 1 \leq i \leq 3s, i \equiv 1 \bmod 2\} \cup \{b_i, d_i \;|\; 1 \leq i \leq 3s, i \equiv 0 \bmod 2\}$ and
\item $\{a_i, c_i \;|\; 1 \leq i \leq 3s, i \equiv 0 \bmod 2\} \cup \{b_i, d_i \;|\; 1 \leq i \leq 3s, i \equiv 1 \bmod 2\}$.
\end{enumerate}
It follows that $\oct(Y_s)=2$.

Furthermore, $Y_s$ is $3$-colourable with colour classes:
\begin{enumerate}
\item $\{x\} \cup \{a_i, d_i \;|\; 1 \leq i \leq 3s, i \equiv 2 \bmod 3\} \cup \{b_i, c_i \;|\; 1 \leq i \leq 3s, i \equiv 1 \bmod 3\}$,
\item $\{y\} \cup \{a_i, d_i \;|\; 1 \leq i \leq 3s, i \equiv 0 \bmod 3\} \cup \{b_i, c_i \;|\; 1 \leq i \leq 3s, i \equiv 2 \bmod 3\}$ and
\item $\{a_i, d_i \;|\; 1 \leq i \leq 3s, i \equiv 1 \bmod 3\} \cup \{b_i, c_i \;|\; 1 \leq i \leq 3s, i \equiv 0 \bmod 3\}$.
\end{enumerate}
In fact, we will show that this $3$-colouring is unique (up to permuting the colours).
To see this, suppose that $c:V(Y_s)\to \{1,2,3\}$ is a $3$-colouring of~$Y_s$.
Since~$x$ and~$y$ are adjacent we may assume without loss of generality that $c(x)=1$ and $c(y)=2$.
Since~$a_1$ and~$d_1$ are adjacent to both~$x$ and~$y$, it follows that $c(a_1)=c(d_1)=3$.
Since~$b_1$ is adjacent to~$y$ and~$a_1$, it follows that $c(b_1)=1$.
By symmetry $c(c_1)=1$.

We prove by induction on~$i$ that for every $i \in \{1,\ldots,3s\}$, $c(a_i)=c(d_i) \equiv i+2 \bmod 3$ and $c(b_i)=c(c_i)\equiv i \bmod 3$.
We have shown that this is true for $i=1$.
Suppose that the claim holds for $i-1$ for some $i \in \{2,\ldots,3s\}$.
Then $c(a_{i-1})=c(d_{i-1}) \equiv (i-1)+2 \bmod 3$ and $c(b_{i-1})=c(c_{i-1})\equiv i-1 \bmod 3$.
Since~$b_i$ is adjacent to~$b_{i-1}$ and~$d_{i-1}$, it follows that~$c(b_i)\equiv i \bmod 3$.
Since~$a_i$ is adjacent to~$b_i$ and~$a_{i-1}$, it follows that $c(a_i)\equiv i+2 \bmod 3$.
By symmetry $c(c_i)\equiv i \bmod 3$ and $c(d_i)\equiv i+2 \bmod 3$.
Therefore the claim holds for~$i$.
By induction, this completes the proof of the claim and therefore shows that the $3$-colouring of~$Y_s$ is indeed unique.

Furthermore, note that the colour classes in this colouring have sizes~$4s+\nobreak 1$, $4s+\nobreak 1$ and~$4s$, respectively.
A set~$S$ is an independent odd cycle transversal of a graph if and only if it is a colour class in some $3$-colouring of this graph.
It follows that $\ioct(Y_s)=4s$.
This completes the proof.\qedllncs
\end{proof}

Before we can prove our main theorem of this section, we need one more lemma, due to Olariu.

\begin{lemma}[\cite{Ol88}]\label{l-ol}
Every connected component of a $\overline{P_1+P_3}$-free graph is either $C_3$-free or complete multi-partite.
\end{lemma}

We are now ready to prove the main result of this section, which immediately implies Theorems~\ref{thm:indep-oct} and~\ref{t-octi}.
If an upper bound given in this theorem is tight, that is, if there exists an $H$-free $3$-colourable graph~$G$ for which equality holds, we again indicate this by a~$*$ in the corresponding row (whereas the other upper bounds are not known to be tight).

\newpage
\begin{theorem}\label{t-oct}
Let~$H$ be a graph.
Then the following two statements hold:
\begin{enumerate}[(i)]
\renewcommand{\theenumi}{(\roman{enumi})}
\renewcommand\labelenumi{(\roman{enumi})}
\item\label{t-oct:i} the class of $H$-free $3$-colourable graphs is $\ioct$-bounded 
\begin{itemize}
\item if~$H$ is an induced subgraph of~$P_4$ or $K_{1,3}+\nobreak sP_1$ for some $s\geq 0$ and
\item only if~$H$ is an induced subgraph of~$K_{1,4}^+$ or $K_{1,4}+\nobreak sP_1$ for some $s\geq 0$.
\end{itemize}
\item\label{t-oct:ii} For $H\notin \{K_{1,3}, K_{1,3}^+, 2P_1+\nobreak P_3\}$, the class of $H$-free $3$-colourable graphs is $\ioct$-identical if and only if~$H$ is a (not necessarily induced) subgraph of~$P_4$ that is not isomorphic to~$2P_2$.\\[-20pt]
\end{enumerate}
In particular, the following statements hold for every $H$-free bipartite graph~$G$: 
\begin{enumerate}[(1)*]
\renewcommand{\theenumi}{(\arabic{enumi})}
\renewcommand\labelenumi{(\arabic{enumi})\phantom{*}}
\refstepcounter{enumi}
\item[\textcolor{darkgray}{\sffamily\bfseries\upshape\mathversion{bold}{(1)*}}]\label{t-oct:1} $\ioct(G)=\oct(G)$ if $H\subseteq P_4$ but $H\neq 2P_2$
\item\label{t-oct:2} $\ioct(G)\leq\oct(G)+s-3$ if $H=sP_1$ for $s\geq 5$
\item\label{t-oct:3} $\ioct(G)\leq\oct(G)+3s-1$ if $H=sP_1+\nobreak P_2$ for $s\geq 3$
\item\label{t-oct:4} $\ioct(G)\leq 2\oct(G)+6s$ if $H=sP_1+\nobreak P_3$ for $s\geq 2$
\item\label{t-oct:5} $\ioct(G)\leq 3\oct(G)$ if $H=K_{1,3}$
\item\label{t-oct:6} $\ioct(G)\leq 3\oct(G)+9s+3$ if $H=K_{1,3}+\nobreak sP_1$ for $s\geq 1$.
\end{enumerate}
\end{theorem}

\begin{proof}
We start by proving~{\sffamily\bfseries\upshape\mathversion{bold}{\ref{t-oct:i}}}.

\subparagraph*{\ref{t-oct:i}: ``$\Leftarrow$''.}
First suppose that~$H$ is an induced subgraph of~$P_4$ or $K_{1,3}+\nobreak sP_1$ for some $s\geq 0$.
Then the class of $H$-free $3$-colourable graphs is $\ioct$-bounded by Lemma~\ref{lem:oct-P4} or Corollary~\ref{cor:oct-K13+sP1}, respectively.

\subparagraph*{\ref{t-oct:i}: ``$\Rightarrow$''.}
Now suppose that the class of $H$-free $3$-colourable graphs is $\ioct$-bounded.
We will prove that~$H$ must be an induced subgraph of~$K_{1,4}^+$ or $K_{1,4}+\nobreak sP_1$ for some $s\geq 0$.
By Lemma~\ref{lem:Cr-fvs-and-oct}, $H$ must be a forest.
By Lemma~\ref{lem:oct-K15}, $H$ must be~$K_{1,5}$-free.
Since~$H$ is a $K_{1,5}$-free forest, it has maximum degree at most~$4$.
By Lemma~\ref{lem:P1+P4-2P2-oct}, $H$ must be $(P_1+\nobreak P_4,2P_2)$-free.

First suppose that~$H$ is $P_4$-free, so every component of~$H$ is a $P_4$-free tree.
Hence every component of~$H$ is a star.
In fact, as~$H$ has maximum degree at most~$4$, every component of~$H$ is an induced subgraph of~$K_{1,4}$.
As~$H$ is $2P_2$-free, at most one component of~$H$ contains an edge.
Therefore~$H$ is an induced subgraph of~$K_{1,4}+\nobreak sP_1$ for some $s\geq 0$ and we are done.

Now suppose that~$H$ contains an induced~$P_4$, say on vertices $x_1,x_2,x_3,x_4$ in that order and let $X=\{x_1,x_2,x_3,x_4\}$.
Since~$H$ is a forest, every vertex $v \in V(H) \setminus X$ has at most one neighbour in~$X$.
A vertex $v \in V(H) \setminus X$ cannot be adjacent to~$x_1$ or~$x_4$, since~$H$ is $2P_2$-free.
By Lemma~\ref{lem:oct-two-deg3}, the vertices~$x_2$ and~$x_3$ cannot both have neighbours outside~$X$; without loss of generality assume that~$x_3$ has no neighbours in $V(H)\setminus X$.
Since~$H$ is $(P_1+\nobreak P_4)$-free, every vertex $v \in V(H) \setminus X$ must have at least one neighbour in~$X$, so it must be adjacent to~$x_2$.
As~$H$ has maximum degree at most~$4$, it follows that~$H$ is an induced subgraph of~$K_{1,4}^+$.
This completes the proof of~\ref{t-oct:i}.

\medskip
\noindent
We now prove {\sffamily\bfseries\upshape\mathversion{bold}{\ref{t-oct:ii}}}.
Let~$H$ be a graph that is not isomorphic to a graph in $\{K_{1,3}, K_{1,3}^+, 2P_1+\nobreak P_3\}$.

\subparagraph*{\ref{t-oct:ii}: ``$\Leftarrow$''.}
First suppose that~$H$ is a subgraph of~$P_4$ that is not isomorphic to~$2P_2$.
If~$H$ is an induced subgraph of~$P_4$, then the claim follows from Lemma~\ref{lem:oct-P4}.
It is sufficient to prove that $\ioct(G)=\oct(G)$ if~$G$ is a $3$-colourable $H$-free graph in three remaining cases, namely when $H=4P_1$, $H=P_1+\nobreak P_3$ and $H=2P_1+\nobreak P_2$.

\subparagraph*{\bf Case~1. $H=4P_1$.\\}
Let~$G$ be a $4P_1$-free $3$-colourable graph and let $X_1$, $X_2$,~$X_3$ be the colour classes of some $3$-colouring of~$G$.
Note that $|X_1|,|X_2|,|X_3| \leq 3$ since~$G$ is $4P_1$-free and $X_1$, $X_2$,~$X_3$ are independent sets.
If $\oct(G) \leq 1$ then $\ioct(G)=\oct(G)$, so we need only consider the case when $\oct(G) \geq 2$.
Since~$G$ is $4P_1$-free, every independent odd cycle transversal has at most three vertices, so $\ioct(G) \leq 3$.

Suppose, for contradiction, that $\oct(G)\neq \ioct(G)$.
Since $\oct(G) \leq \ioct(G)$, it follows that $\oct(G)=2$ and $\ioct(G)=3$.
If $|X_i| < 3$ for some $i \in \{1,2,3\}$ then~$X_i$ is an independent odd cycle transversal on fewer than three vertices, a contradiction.
It follows that $|X_1|=|X_2|=|X_3|=3$ and so~$G$ has exactly nine vertices.
Let~$S$ be a minimum odd cycle transversal of~$G$, in which case $|S|=2$.
Then $G-S$ is a bipartite graph on seven vertices.
Therefore one of the parts of $G-S$ contains at least four vertices, and so $G-S$ (and therefore~$G$) contains an induced~$4P_1$.
This contradiction implies that $\ioct(G)=\oct(G)$.

\subparagraph*{\bf Case 2. $H=P_1+\nobreak P_3$.\\}
Let~$G$ be a $(P_1+\nobreak P_3)$-free $3$-colourable graph and let $X_1$, $X_2$,~$X_3$ be the colour classes of some $3$-colouring of~$G$.
By Lemma~\ref{l-ol}, every connected component of a $\overline{P_1+P_3}$-free graph is either $C_3$-free or complete multi-partite.
Let $D_1,\ldots,D_r$ be the connected components of~$\overline{G}$.
Then~$V(G)$ can be partitioned into sets $A_1,\ldots,A_r$, with $A_i=V(D_i)$ for $i\in \{1,\ldots,r\}$, such that 
\begin{enumerate}[(a)]
\renewcommand{\theenumi}{(\alph{enumi})}
\renewcommand\labelenumi{(\alph{enumi})}
\item\label{prop:a} for all $i \in \{1,\ldots,r\}$, the graph~$G[A_i]$ is either $3P_1$-free or a disjoint union of complete graphs, and
\item\label{prop:b} for all $i,j \in \{1,\ldots,r\}$ with $i\neq j$, the set~$A_i$ is complete to the set~$A_j$.
\end{enumerate}
As~$G$ is $3$-colourable and hence contains no~$K_4$, Property~\ref{prop:b} implies that $r\leq 3$.
First suppose that $r=3$.
Then, as~$G$ is $3$-colourable, each~$A_i$ must be an independent set.
Hence, $G$ is a complete $3$-partite graph with partition classes $A_1$, $A_2$,~$A_3$.
It follows that $\ioct(G)=\oct(G)=\min\{|A_1|,|A_2|,|A_3|\}$.

Now suppose that $r=2$.
As~$G$ is $3$-colourable and~$A_1$ is complete to~$A_2$, one of the sets~$A_1$ or~$A_2$, say~$A_1$, must be an independent set, and the other set, $A_2$, must induce a bipartite graph.
First assume that~$G[A_2]$ is a disjoint union of complete graphs.
As~$G[A_2]$ is bipartite, this means that every connected component of~$G[A_2]$ has at most two vertices (see \figurename~\ref{f-a1a2} for an example).
\begin{figure}
\begin{center}
\begin{tikzpicture}[xscale=1, yscale=1]
\draw (2.5,-1.5) -- (1.5,-1.5) (5.5,-1.5) -- (6.5,-1.5)
(5.5,1.5) -- (1.5,-1.5) (5.5,1.5) -- (2.5,-1.5) (5.5,1.5) -- (3.5,-1.5) (5.5,1.5) -- (4.5,-1.5) (5.5,1.5) -- (5.5,-1.5) (5.5,1.5) -- (6.5,-1.5)
(4.5,1.5) -- (1.5,-1.5) (4.5,1.5) -- (2.5,-1.5) (4.5,1.5) -- (3.5,-1.5) (4.5,1.5) -- (4.5,-1.5) (4.5,1.5) -- (5.5,-1.5) (4.5,1.5) -- (6.5,-1.5)
(3.5,1.5) -- (1.5,-1.5) (3.5,1.5) -- (2.5,-1.5) (3.5,1.5) -- (3.5,-1.5) (3.5,1.5) -- (4.5,-1.5) (3.5,1.5) -- (5.5,-1.5) (3.5,1.5) -- (6.5,-1.5)
(2.5,1.5) -- (1.5,-1.5) (2.5,1.5) -- (2.5,-1.5) (2.5,1.5) -- (3.5,-1.5) (2.5,1.5) -- (4.5,-1.5) (2.5,1.5) -- (5.5,-1.5) (2.5,1.5) -- (6.5,-1.5)
(1.5,2) -- (1,2) -- (1,1) -- (1.5,1)
(1.5,-2) -- (1,-2) -- (1,-1) -- (1.5,-1);
\draw[fill=white] circle [radius=3pt] (2.5,1.5) circle [radius=3pt] (3.5,1.5) circle [radius=3pt] (4.5,1.5) circle [radius=3pt] (5.5,1.5)
circle [radius=3pt] (1.5,-1.5) circle [radius=3pt] (2.5,-1.5) circle [radius=3pt] (3.5,-1.5) circle [radius=3pt] 
(4.5,-1.5) circle [radius=3pt] (5.5,-1.5) circle [radius=3pt] (6.5,-1.5) circle [radius=3pt] (6.5,-1.5);
\filldraw[white] (-0.2,-0.2) rectangle (0.2,0.2);
\node[left] at (1,1.5) {$A_1$};
\node[left] at (1,-1.5) {$A_2$};
\end{tikzpicture}
\end{center}
\caption{An example of a $(P_1+\nobreak P_3)$-free $3$-colourable graph~$G$ in the case when $r=2$ and~$G[A_2]$ is the disjoint union of one or more complete graphs on at most two vertices.}\label{f-a1a2}
\end{figure}
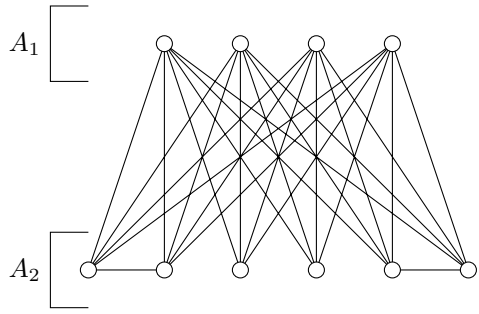
Pick a vertex of each edge in~$G[A_2]$ and let~$A_2'$ be the set of selected vertices.
Then $\ioct(G)=\oct(G)=\min\{|A_1|,|A_2'|\}$.
By Property~\ref{prop:a}, it remains to consider the case when~$G[A_2]$ is bipartite and $3P_1$-free.
Then $\ivc(G[A_2]) \leq 2$ and so $\ioct(G) \leq 2$ and therefore $\ioct(G)=\oct(G)$.

Finally, suppose that $r=1$.
If $G=G[A_1]$ is the disjoint union of complete graphs, then each complete graph must have at most three vertices (as~$G$ is $3$-colourable).
This implies that $\ioct(G)=\oct(G)$.
If $G=G[A_1]$ is $3P_1$-free, then $\ioct(G) \leq 2$ and therefore $\ioct(G)=\oct(G)$.
We conclude that $\ioct(G)=\oct(G)$.

\subparagraph*{Case 3. $H=2P_1+\nobreak P_2$.\\}
Let~$G$ be a $(2P_1+\nobreak P_2)$-free $3$-colourable graph.
As~$G$ is $3$-colourable, we can partition~$V(G)$ into three independent sets $A$, $B$,~$C$.
If $\oct(G)\leq 1$, then $\ioct(G)=\oct(G)$.
Hence, we may assume that $\oct(G)\geq 2$.
For contradiction, we assume that $\ioct(G)\geq \oct(G)+\nobreak 1$.

As $\oct(G)\geq 2$, it follows that~$G$ is not bipartite.
Hence, $A$, $B$,~$C$ are non-empty and moreover, there exists an edge between each pair of these sets.
We claim that every subgraph of~$G$ induced by two vertices in one set in $\{A,B,C\}$ and two vertices in another set in $\{A,B,C\}$ has at least one edge.
This can be seen as follows.
For contradiction, suppose that there exist two vertices $a_1$, $a_2$ of~$A$ and two vertices $b_1$, $b_2$ of~$B$, such that $\{a_1,a_2,b_1,b_2\}$ is an independent set.
As $G[A\cup B]$ contains an edge, there exist adjacent vertices $x\in A$ and $y\in B$.
As $\{a_1,a_2,b_1,b_2\}$ is an independent set, it follows that $x\notin \{a_1,a_2\}$ or $y\notin \{b_1,b_2\}$
Assume without loss of generality that $x\notin \{a_1,a_2\}$.
Then~$y$ must be adjacent to least one of $a_1,a_2$, as otherwise $\{a_1,a_2,x,y\}$ would induce~$2P_1+\nobreak P_2$.
Assume without loss of generality that~$y$ is adjacent to~$a_1$.
Then $y\notin \{b_1,b_2\}$, as $\{a_1,a_2,b_1,b_2\}$ is an independent set.
However, now $\{b_1,b_2,a_1,y\}$ induces~$2P_1+\nobreak P_2$, a contradiction.
Hence, the claim holds.

Now let~$S$ be a minimum odd cycle transversal of~$G$.
Let $A'=A\setminus S$, $B'=B\setminus S$ and $C'=C\setminus S$.
First suppose that each of $A'$, $B'$,~$C'$ contains at least three vertices.
As~$S$ is an odd cycle transversal, $G-S=G[A'\cup B'\cup C']$ is bipartite.
Hence, $A'\cup B'\cup C'$ can be partitioned into two independent sets~$X$ and~$Y$.
As each of $A'$, $B'$,~$C'$ has at least three vertices, one of $X$,~$Y$, say~$X$, contains two vertices of at least two sets of $A'$, $B'$,~$C'$.
By the above claim, $G[X]$ contains an edge, a contradiction.
Hence, we may assume without loss of generality that $|A'|\leq 2$, so $|S\cap A|\geq |A|-2$.
Since~$A$ is an independent odd cycle transversal, it follows that $|A|\geq \ioct(G)$.
Hence, we obtain $$|S\cap A| \geq |A|-2\geq \ioct(G)-2 \geq \oct(G)-1=|S|-1.$$
As~$S$ is not an independent set, this implies that $|S\cap A| = |A|-2=|S|-1$, and thus~$S$ contains exactly one vertex from $B\cup C$, say, $S\cap B=\{b\}$ (and thus $S\cap C=\emptyset$).
As $|S\cap A|=|A|-2$, it follows that $|A'|=|A\setminus S|=2$.
Let $A'=\{a',a''\}$.
Since $\ioct(G) > \oct(G) \geq 2$, and $B$~and~$C$ are odd cycle transversals, it follows that $|B|,|C| \geq 3$.

Suppose that $|B|\geq 4$.
As $\ioct(G) > \oct(G)$, the independent set $(A \cap S) \cup \{ a'' \}$ is not an odd cycle transversal.
Consequently, $G - ((A \cap S) \cup \{ a'' \})=G[\{a'\}\cup B\cup C]$ is not bipartite.
As $G[B\cup C]$ is bipartite, this means that $G - ((A \cap S) \cup \{ a'' \})$ has an odd cycle containing~$a'$.
This implies that~$a'$ has a neighbour in both~$B$ and~$C$.
As~$G$ is $(2P_1+\nobreak P_2)$-free and $|B|\geq 4$, this means that~$a'$ has at least three neighbours in~$B$, and thus at at least two neighbours $b_1$,~$b_2$ in $B \setminus \{ b \}$.
As $|C|\geq 3$, we find for the same reason that~$a'$ has at least two neighbours $c_1$,~$c_2$ in~$C$.
By our previous claim, there is at least one edge with one end-vertex in $\{b_1,b_2\}$, say~$b_1$, and the other one in $\{c_1,c_2\}$, say~$c_1$.
However, now $\{a',b_1,c_1\}$ induces a~$C_3$ in $G- ((A \cap S) \cup \{ b \})$, contradicting the fact that $S=(A \cap S) \cup \{ b \}$ is an odd cycle transversal.
We conclude that $|B|=3$, say $B=\{b,b',b''\}$.

As $3=|B|\geq \ioct(G)>\oct(G)=|S|\geq 2$, we find that $|S| = 2$.
Hence $|S \cap A|=1$ and $|A|=|S|+2=3$, say $S=\{a,b\}$ and $A=\{a,a',a''\}$.
In particular, both~$a'$ and~$a''$ are adjacent to at least one vertex of~$B$ and to at least one vertex of~$C$, as otherwise $\{a,a''\}$ or $\{a,a'\}$, respectively, is an independent odd cycle transversal of~$G$ of size~$2$.

By our claim, there exists at least one edge between a vertex of $\{a',a''\}$, say~$a'$, and a vertex of $\{b',b''\}$, say~$b'$.
Since $\{ b, b'' \}$ is not an odd cycle transversal and $G[A\cup C]$ is bipartite, $b'$ belongs to an odd cycle in $G-\{b,b''\}=G[A\cup C\cup \{b'\}]$.
This implies that~$b'$ has a neighbour in~$C$.
This, together with the fact that~$G$ is $(2P_1+\nobreak P_2)$-free, implies that~$b'$ is adjacent to all but at most one vertex in~$C$.
Recall that~$a'$ also has a neighbour in~$C$.
By the same argument, this means that~$a'$ is adjacent to all but at most one vertex in~$C$.
Since $|C| \geq 3$, we find that~$a'$ and~$b'$ have a common neighbour~$c\in C$.
Then, as~$a'$ and~$b'$ are adjacent, $\{a',b',c\}$ induces a~$C_3$ in $G-\{a,b\}$, contradicting the fact that $S=\{a,b\}$ is an odd cycle transversal of~$G$.
We conclude that $\ioct(G)=\oct(G)$.

\subparagraph*{\ref{t-oct:ii}: ``$\Rightarrow$''.}
Now suppose that $H=2P_2$ or~$H$ is not a subgraph of~$P_4$.
By~\ref{t-oct:i} we may assume that~$H$ is an induced subgraph of~$K_{1,4}^+$ or $K_{1,4}+\nobreak sP_1$ for some $s\geq 0$, which in particular implies that  $H \neq 2P_2$.
Recall that $H\notin  \{K_{1,3},K_{1,3}^+,2P_1+\nobreak P_3\}$.
This means that~$H$ contains an induced subgraph from the set $\{K_{1,4}, K_{1,3}+\nobreak P_1,5P_1,3P_1+\nobreak P_2\}$.

\begin{figure}\label{oct1}
\begin{center}
\begin{tikzpicture}[xscale=0.8, yscale=0.8]
\draw (1,0) -- (-1,0) (1,0) -- (3,1.5) -- (2.2,2.3) -- (1,0) (1,0) -- (3,-1.5) -- (2.2,-2.3) -- (1,0) 
(-1,0) -- (-3,1.5) -- (-2.2,2.3) -- (-1,0) (-1,0) -- (-3,-1.5) -- (-2.2,-2.3) -- (-1,0);
\draw[fill=white] (1,0) circle [radius=3pt] (3,1.5) circle [radius=3pt] (3,-1.5) circle [radius=3pt] (-3,1.5) circle [radius=3pt] 
(-2.2,2.3) circle [radius=3pt] (-3,-1.5) circle [radius=3pt] (-2.2,-2.3) circle [radius=3pt];
\draw[fill=black]  (-1,0) circle [radius=3pt] (2.2,-2.3) circle [radius=3pt] (2.2,2.3) circle [radius=3pt]; 
\node[above right] at (-1,0) {$u$};
\node[above left] at (1,0) {$v$};
\node[left] at (2.2,2.3) {$u_1$};
\node[left] at (2.2,-2.3) {$u_2$};
\end{tikzpicture}
\end{center}
\caption{A $(K_{1,4},K_{1,3}+\nobreak P_1,5P_1)$-free $3$-colourable graph~$G$ with~$\ioct(G)=\oct(G)+\nobreak 1$.}\label{f-counter}
\end{figure}

First consider the graph~$G$ from \figurename~\ref{f-counter}.
It is readily seen that~$G$ is $(K_{1,4}, K_{1,3}+\nobreak P_1,5P_1)$-free and $3$-colourable.
Moreover, $\{u,v\}$ is a minimum odd cycle transversal, so $\oct(G)=2$, while $\ioct(G)=3$ (for instance, $\{u,u_1,u_2\}$ is a minimum independent odd cycle transversal of~$G$).
Now consider the graph~$G$ from \figurename~\ref{f-counter2}.
It is readily seen that~$G$ is $(3P_1+\nobreak P_2)$-free and $3$-colourable.
Moreover, $\oct(G)=2$, as $\{u,v\}$ is a minimum odd cycle transversal, while $\ioct(G)=3$ (for instance, $\{u,u_1,u_2\}$ is a minimum independent odd cycle transversal of~$G$).
This completes the proof of~\ref{t-oct:ii}.

\begin{figure}
\begin{center}
\begin{tikzpicture}[xscale=0.8, yscale=0.8]
\draw (1,0) -- (2,1) -- (4,1) -- (1,0) -- (2,-1) -- (4,-1) -- (1,0) (2,1) -- (2,-1) (4,1) -- (4,-1)
(-1,0) -- (-2,1) -- (-4,1) -- (-1,0) -- (-2,-1) -- (-4,-1) -- (-1,0) (-2,1) -- (-2,-1) (-4,1) -- (-4,-1) (1,0) -- (-1,0);
\draw[fill=white] 
(1,0) circle [radius=3pt] (4,1) circle [radius=3pt] (2,-1) circle [radius=3pt] (-4,1) circle [radius=3pt] (-2,1) circle [radius=3pt] (-2,-1) circle [radius=3pt] (-4,-1) circle [radius=3pt];
\draw[fill=black] (-1,0) circle [radius=3pt] (2,1) circle [radius=3pt] (4,-1) circle [radius=3pt];
\node[above right] at (-1,0) {$u$};
\node[above left] at (1,0) {$v$};
\node[left] at (2,1) {$u_1$};
\node[right] at (4,-1) {$u_2$};
\end{tikzpicture}
\end{center}
\caption{A $(3P_1+\nobreak P_2)$-free $3$-colourable graph~$G$ with $\ioct(G)=\oct(G)+\nobreak 1$.}\label{f-counter2}
\end{figure}

\medskip
\noindent
We now consider Statements~\ref{t-oct:1}--\ref{t-oct:6}.
Statement~\ref{t-oct:1} immediately follows from Statement~\ref{t-oct:ii}, whereas Lemma~\ref{lem:oct-K13} and  Corollary~\ref{cor:oct-K13+sP1} imply Statements~\ref{t-oct:5} and~\ref{t-oct:6}, respectively.
It remains to prove Statements~\ref{t-oct:2}--\ref{t-oct:4}.

\subparagraph*{\ref{t-oct:2}.}
Let $H=sP_1$ for some $s\geq 5$.
Let~$G$ be an $sP_1$-free $3$-colourable graph.
If $\oct(G)\leq 1$, then $\ioct(G)=\oct(G)$.
Hence, we may assume that $\oct(G)\geq 2$.
As~$G$ is $3$-colourable, $V(G)$ can be partitioned into three independent sets $V_1$, $V_2$,~$V_3$.
Hence, $V_1$ is an independent odd cycle transversal.
As~$G$ is $sP_1$-free, $V_1$ has size at most~$s-1$.
This means that $\ioct(G)\leq s-1 = 2 + s-3 \leq \oct(G)+s-3$.

\newpage
\subparagraph*{\ref{t-oct:3} and~\ref{t-oct:4}.}
Statements~\ref{t-oct:3} and~\ref{t-oct:4} follow from Lemma~\ref{lem:oct-K_1r-to-K13+sP_1} after observing that $\ioct(G)=\oct(G)$ holds for every $K_{1,r}$-free $3$-colourable graph $G$ with $r\in\{1,2\}$ (this also follows from~\ref{t-oct:1}).
This completes the proof.\qedllncs
\end{proof}

\section{Conclusions}\label{s-con}
To develop an insight into the price of independence for classical concepts, we have investigated whether or not the size of a minimum independent vertex cover, feedback vertex set or odd cycle transversal is bounded in terms of the minimum size of the not-necessarily-independent variant of each of these transversals for $H$-free graphs (that have such independent transversals).
While we note that the bounds we give in some of our results are tight, in this paper we were mainly concerned with obtaining dichotomy results on whether there is a bound, rather than trying to find exact bounds.
We will now discuss some open problems resulting from our work.

We fully classified for which graphs~$H$ the class of $H$-free bipartite graphs is $\ivc$-bounded and for which graphs~$H$ the class of $H$-free near-bipartite graphs is $\ifvs$-bounded.
By Lemma~\ref{lem:oct-K_1r-to-K13+sP_1}, for $r,s \geq 1$ the class of $K_{1,r}$-free $3$-colourable graphs is $\ioct$-bounded if and only if the class of $(K_{1,r}+\nobreak sP_1)$-free $3$-colourable graphs is $\ioct$-bounded.
Therefore, Theorem~\ref{thm:indep-oct} (and similarly, Theorem~\ref{t-oct}\ref{t-oct:i}) leaves three open cases with respect to $\ioct$-boundedness, as follows:

\begin{openproblem}\label{oprob:oct}
Determine whether the class of $H$-free $3$-colourable graphs is $\ioct$-bounded when~$H$ is:
\begin{enumerate}
\item $K_{1,4}$ (or equivalently $K_{1,4}+\nobreak sP_1$ for any $s \geq 1$),
\item $K_{1,3}^+$ or
\item $K_{1,4}^+$.
\end{enumerate}
\end{openproblem}

\noindent
We fully classified for which graphs~$H$ the class of $H$-free bipartite graphs is $\ivc$-identical.
However, we have a few remaining cases for the notions of being $\ifvs$-identical (one open case) and being $\ioct$-identical (three open cases):

\begin{openproblem}\label{o-1}
Does there exist a $K_{1,3}$-free near-bipartite graph~$G$ with $\ifvs(G)>\fvs(G)$?
\end{openproblem}

\begin{openproblem}\label{o-2}
For $H\in \{K_{1,3},K_{1,3}^+,2P_1+\nobreak P_3\}$, does there exist an $H$-free $3$-colourable graph~$G$ with $\ioct(G)>\oct(G)$?
\end{openproblem}
In particular, we note that the $H=K_{1,3}^+$ case is the only one open for both Open Problem~\ref{oprob:oct} and Open Problem~\ref{o-2}.
We also note that, in contrast to the class of $(2P_1+\nobreak P_3)$-free $3$-colourable graphs (see, for example,~\cite{BGPS12}), the classes of $K_{1,3}$-free near-bipartite graphs and $K_{1,3}$-free $3$-colourable graphs are \NP-complete to recognize.
This follows from the results that the problems of deciding near-bipartiteness~\cite{BDFJP17b} and deciding $3$-colourability~\cite{Ho81} are \NP-complete for line graphs, which form a subclass of $K_{1,3}$-free graphs.

\medskip
\noindent
As results for the price of connectivity implied algorithmic consequences for connected transversal problems~\cite{CHJMP18,JPP18}, it is natural to ask whether our results for the price of independence have similar consequences.
The problems {\sc Independent Vertex Cover}, {\sc Independent Feedback Vertex Set} and {\sc Independent Odd Cycle Transversal} ask to determine the minimum size of the corresponding independent transversal.
The first problem is readily seen to be polynomial-time solvable.
The other two problems are \NP-hard for $H$-free graphs whenever~$H$ is not a linear forest~\cite{BDFJP17b}, just like their classical counterparts {\sc Feedback Vertex Set}~\cite{MPRS12,Mu17} and {\sc Odd Cycle Transversal}~\cite{EHK98} (see also~\cite{KKTW01,KL07}).
The complexity of these four problems restricted to $H$-free graphs is still poorly understood when~$H$ is a linear forest.
Our results suggest that it is unlikely that we can obtain polynomial algorithms for the independent variants based on results for the original variants, as the difference between~$\ifvs(G)$ and~$\fvs(G)$ and between~$\ioct(G)$ and~$\oct(G)$ can become unbounded quickly.

\bibliography{poireferences}

\end{document}